\documentclass[11pt, reqno]{amsart}

\usepackage{amsmath, amsthm, amscd, amsfonts, amssymb, graphicx, color, mathtools, mathrsfs}
\usepackage[bookmarksnumbered, colorlinks, plainpages]{hyperref}
\usepackage{enumerate}
\usepackage{setspace}
\usepackage{multicol}
\usepackage[margin=1in]{geometry}
\usepackage{comment}
\usepackage{booktabs}
\usepackage{caption}
\usepackage{subcaption}
\usepackage{pgfplots}

\theoremstyle{definition}
\newtheorem{theorem}{Theorem}[section]
\newtheorem{lemma}[theorem]{Lemma}
\newtheorem{proposition}[theorem]{Proposition}
\newtheorem{algorithm}[theorem]{Algorithm}

\newtheorem{remark}[theorem]{Remark}
\numberwithin{equation}{section}

\newcommand{\bff}{\boldsymbol}
\newcommand{\bb}{\mathbb}

\newcommand{\dt}{\mathrm{d}t}
\newcommand{\dtt}{\mathrm{d}_t}
\newcommand{\Dtt}{\mathcal{D}_t}
\newcommand{\ddt}{\frac{\mathrm{d}}{\mathrm{d}t}}
\newcommand{\dx}{\mathrm{d}x}

\newcommand{\norm}[2]{\left\|{#1}\right\|_{#2}}
\newcommand{\inpro}[2]{\left\langle#1,#2\right\rangle}
\newcommand{\abs}[1]{\left|{#1}\right|}
\newcommand{\F}[1]{\mathcal{F}[{#1}]}

\allowdisplaybreaks

\begin{document}
	\setcounter{page}{1}
	
	\title[Error analysis of SAV FEM for the Landau--Lifshitz--Bloch equation]
	{Error analysis of scalar auxiliary variable finite element methods for the Landau--Lifshitz--Bloch equation}
	
	\author[Agus L. Soenjaya]{Agus L. Soenjaya}
	\address{School of Mathematics and Statistics, The University of New South Wales, Sydney 2052, Australia}
	\email{\textcolor[rgb]{0.00,0.00,0.84}{a.soenjaya@unsw.edu.au}}

	\keywords{Landau--Lifshitz--Bloch, scalar auxiliary variable, finite element method, BDF2, error estimates, micromagnetics}
	\subjclass{35K59, 35Q60, 65M12, 65M60}
		
	\date{\today}
	
	\begin{abstract}
	The Landau--Lifshitz--Bloch (LLB) equation is a well-established micromagnetic model for describing magnetisation dynamics in ferromagnets at elevated temperatures. In this paper, we propose and analyse two fully discrete, conforming finite element schemes based on the scalar auxiliary variable (SAV) approach for solving the LLB equation in the high-temperature regime above the Curie point. The first scheme employs a semi-implicit Euler time discretisation, while the second is based on a linearly extrapolated BDF2 method. Both schemes are linear, unconditionally stable with respect to the energy norm, and satisfy a discrete energy law involving an SAV-based energy functional that approximates the true micromagnetic energy. Under suitable regularity assumptions, we establish unconditional energy stability and derive optimal-order error estimates in $\mathbb{L}^2, \mathbb{H}^1$, and $\mathbb{L}^\infty$ norms. To the best of our knowledge, this is the first rigorous error analysis of a fully discrete SAV-based method for a quasilinear vector-valued problem, as well as the first linear, energy-stable scheme for the LLB equation in the high-temperature regime that achieves second-order temporal accuracy. Numerical experiments validate the theoretical convergence rates and compare the efficiency of the proposed schemes with existing methods; additional tests with adaptive time stepping illustrate the behaviour of the SAV-based modified energy in relation to the physical energy.
	\end{abstract}
	\maketitle
	\tableofcontents
	
\section{Introduction}

The Landau--Lifshitz--Bloch (LLB) equation is a fundamental model for describing the dynamics of magnetisation in ferromagnetic materials at elevated temperatures~\cite{AtxHinNow16, ChuNie20, ChuNowChaGar06, Gar97}, particularly near or above the Curie temperature, where the standard Landau--Lifshitz model becomes inadequate. The LLB equation incorporates both transverse and longitudinal fluctuation mechanisms, making it well-suited for modelling thermally induced magnetisation dynamics in heat-assisted magnetic recording (HAMR), for instance. Its nonlinear structure, however, presents considerable analytical and numerical challenges. Therefore, the development and analysis of stable and accurate numerical methods for solving the LLB equation are crucial for both theoretical understanding and practical applications in computational micromagnetics.

We will introduce the model next. Let $\mathscr{D}$ be a bounded domain. The dynamics of the magnetisation $\bff{u}:[0,T]\times \mathscr{D}\to \bb{R}^3$ under the influence of an effective field $\bff{H}:[0,T]\times \mathscr{D}\to \bb{R}^3$ above the Curie temperature is given by the LLB equation, which is a vector-valued second-order quasilinear PDE~\cite{Le16}. The problem reads:
\begin{subequations}\label{equ:llb a}
	\begin{alignat}{2}
		\label{equ:llb eq1}
		&\partial_t \bff{u}
		=
		-
		\gamma \bff{u} \times \bff{H}
		+
		\alpha \bff{H}
		\,
		\qquad && \text{for $(t,\bff{x})\in(0,T)\times\mathscr{D}$,}
		\\
		\label{equ:llb eq2}
		&\bff{H}=\sigma \Delta \bff{u} - \kappa\mu \bff{u} - \kappa |\bff{u}|^2 \bff{u}
		\qquad && \text{for $(t,\bff{x})\in(0,T)\times\mathscr{D}$,}
		\\
		\label{equ:llb init}
		&\bff{u}(0,\bff{x})= \bff{u}_0(\bff{x}) 
		\qquad && \text{for } \bff{x}\in \mathscr{D},
		\\
		\label{equ:llb bound}
		&\displaystyle{
			\frac{\partial \bff{u}}{\partial \bff{n}}= \bff{0}}
		\qquad && \text{for } (t,\bff{x})\in (0,T) \times \partial \mathscr{D},
	\end{alignat}
\end{subequations}
where $\gamma>0$ is the electron gyromagnetic ratio, $\alpha>0$ is the damping coefficient, $\sigma>0$ is the exchange damping coefficient, $\kappa=(2\chi)^{-1}$ is a positive constant related to the longitudinal susceptibility $\chi$ of the ferromagnet, and $\mu>0$ is a constant related to the equilibrium magnetisation magnitude.

The effective field $\bff{H}$ is the negative variational derivative of the micromagnetic energy functional $\mathcal{E}$ defined by
\begin{align}\label{equ:energy}
	\mathcal{E}[\bff{u}] := \int_{\mathscr{D}} \left(\frac{\sigma}{2} \abs{\nabla \bff{u}}^2 + \frac{\kappa\mu}{2} \abs{\bff{u}}^2 + \frac{\kappa}{4} (\abs{\bff{u}}^4+1) \right) \dx.
\end{align}
In \eqref{equ:energy}, we consider contributions from the exchange energy and the Ginzburg--Landau internal energy, leading to the effective field as stated in \eqref{equ:llb eq2}. Other lower order terms, such as the anisotropy energy and the external magnetic energy, could be taken into account without difficulties, but are omitted for simplicity of presentation. It is known that a system described by \eqref{equ:llb a} dissipates energy over time~\cite{LeSoeTra24, Soe25}.

We now review relevant mathematical results for the LLB equation in the regime above the Curie temperature. The existence of a global weak solution to \eqref{equ:llb a} in smooth bounded domains is established in~\cite{Le16}, while the unique existence of a global strong solution to \eqref{equ:llb a} in polytopal domains is proven in \cite{LeSoeTra24}. The existence of a global smooth solution in an unbounded domain, under an additional smallness condition on the initial data when $d=3$, is shown in~\cite{PuYan22}.

Several numerical schemes for solving \eqref{equ:llb a} have been proposed in the literature. An energy-stable, but \emph{nonlinear}, finite element method for a \emph{regularised} version of \eqref{equ:llb a} is proposed and analysed in \cite{Soe24}, where optimal order convergence is demonstrated. This regularised problem includes the Landau--Lifshitz--Baryakhtar equation~\cite{Bar84, SoeTra23}, another micromagnetic model at elevated temperatures. A different regularisation of the LLB equation is introduced in \cite{LeSoeTra24}, where a linear, though not necessarily energy-stable, finite element scheme is proposed. In~\cite{GolJiaLe25}, a simpler biharmonic regularisation was employed in the numerical treatment of the stochastic LLB equation to improve the regularity of the solution; see also~\cite{Soe26} for related results.

An optimally convergent finite element method (which is either linear or energy-stable, but \emph{not} both) to approximate the solution of \eqref{equ:llb a} directly is later proposed and analysed in \cite{Soe25}; see also \cite{Soe25c} for the case below the Curie temperature. A conditionally stable, nonlinear scheme is also proposed in~\cite{BenEssAyo24}, where convergence towards weak solutions (without rate) is shown conditionally, but subject to a severe time-step constraint. On a related note, a conditionally stable, decoupled, linearised finite element scheme is proposed in~\cite{Soe26s} to solve the LLB equation coupled with nonlinear spin diffusion equation.

The aim of this paper is twofold. First, we propose fully discrete, conforming finite element schemes for the LLB equation \eqref{equ:llb a} that are both linear and provably stable in the energy norm by employing the scalar auxiliary variable (SAV) approach~\cite{SheXuYan18, SheXuYan19}. Second, we rigorously analyse the stability and convergence of these schemes, proving optimal-order error estimates under suitable regularity assumptions on the solution. Notably, the second scheme introduced in this paper employing BDF2 time discretisation attains second-order accuracy in time (assuming sufficiently regular exact solution), in contrast to existing schemes in the literature which are limited to first-order temporal accuracy. 

Recently, the SAV methodology has gained wide popularity for designing energy-stable schemes for various gradient flows and semilinear PDEs~\cite{CheMaoShe20, HouQia23}, including the Cahn--Hilliard equation~\cite{CheMaoShe20}, the Navier--Stokes equation~\cite{LiSheLiu21, ZhaYua22}, the phase field crystal equation~\cite{WanHuaWan21}, the Navier--Stokes--Cahn--Hilliard system~\cite{LiShe20, YanYiChe24}, and the magnetohydrodynamic system~\cite{LiWanShe22}. This list is by no means exhaustive.

It is worth noting that a common feature of the systems mentioned above is that they are \emph{semilinear}. In contrast, the LLB equation~\eqref{equ:llb a} is \emph{quasilinear} and is \emph{not} a gradient flow due to the term $-\gamma\bff{u}\times \bff{H}$, which complicates the analysis. Furthermore, the unknowns $\bff{u}$ and $\bff{H}$ are vector-valued, precluding the use of mean value theorems commonly employed in the analysis of SAV-based schemes for scalar-valued gradient flows. We overcome these inherent difficulties and derive optimal order convergence by treating the cross product term semi-implicitly, employing the Ritz projection, and carefully applying discrete Gagliardo--Nirenberg inequalities and Sobolev embeddings in the derivation of key superconvergence estimates. To the best of our knowledge, this is the first rigorous error analysis of a fully-discrete SAV-based scheme applied to a quasilinear system.

An important feature of \eqref{equ:llb a} is the dissipation of micromagnetic energy \eqref{equ:energy} over time, and it is essential to maintain this property at the discrete level.
As discussed in \cite{Soe25}, finite element schemes that are not provably energy-stable may produce non-physical solutions with non-decreasing energy. {The SAV approach proposed in this paper enables the construction of linear, unconditionally energy-stable schemes by reformulating the nonlinear energy into a quadratic form via an auxiliary variable. In addition, the quasilinear cross-product term is discretised in a semi-implicit, structure-preserving manner. Specifically, it is reformulated in terms of the effective field and treated in a weak form that is consistent with the energy functional. This allows one to exploit a discrete orthogonality property of the cross-product term, ensuring that it does not contribute to the energy variation at the discrete level. As a result, the scheme satisfies a discrete dissipation law for a modified energy functional, which serves as a surrogate Lyapunov functional and consistently approximates the original micromagnetic energy. This structure yields uniform stability in the energy norm (i.e., the $\bb{H}^1$-norm), enabling control of the gradient of the numerical solution. Such stability is generally not available for schemes for the LLB equation that lack a discrete energy law~\cite{BenEssAyo24, GolJiaLe25, LeSoeTra24, Soe25c} and plays a crucial role in the subsequent error analysis.}

This paper is organised as follows. Notations and preparatory results employed in the paper are collected in Section~\ref{sec:prelim}. Error analysis for the SAV FEM with linearised Euler and linearised BDF2 time discretisation are established in Section~\ref{sec:sav fem euler} and \ref{sec:sav fem cn}, respectively. The main results of this paper are Theorem~\ref{the:err euler} and Theorem~\ref{the:err cn}, where optimal-order error estimates in the $\bb{L}^2$ and $\bb{H}^1$ norms, as well as in the $\bb{L}^\infty$ norm (for $d\leq 2$), are proven. Numerical experiments are presented in Section~\ref{sec:exp} to verify the theoretical convergence rates, compare the performance of the schemes with existing methods, and investigate the energy dynamics, including the effect of adaptive time stepping.

\section{Preliminaries}\label{sec:prelim}

Some notations, SAV formulations of the problem, and essential details of the finite element approximation employed in this paper are outlined in the following subsections.

\subsection{Notations}
Some notations used throughout this paper are defined in this section. The function space $\bb{L}^p := \bb{L}^p(\mathscr{D}; \bb{R}^3)$ denotes the usual space of $p$-th integrable functions taking values in $\bb{R}^3$ and $\bb{W}^{s,p} := \bb{W}^{s,p}(\mathscr{D}; \bb{R}^3)$ denotes the usual Sobolev space of 
functions on $\mathscr{D} \subset \bb{R}^d$ taking values in $\bb{R}^3$. We
write $\bb{H}^s := \bb{W}^{s,2}$ and set $\bb{W}^{0,p}=\bb{L}^p$. Here, $\mathscr{D}\subset \bb{R}^d$ for $d=1,2,3$
is an open and bounded convex polytopal domain. The Neumann Laplacian operator acting on $\bb{R}^3$-valued functions is denoted by $\Delta$.

If $X$ is a Banach space, the spaces $L^p(0,T; X)$ and $W^{s,p}(0,T;X)$ denote respectively the usual Lebesgue and Sobolev spaces of functions on $(0,T)$ taking values in $X$. For simplicity, we write $W^{s,p}_T(\bb{W}^{m,r}) := W^{s,p}(0,T; \bb{W}^{m,r})$ and $L^p_T(\bb{L}^q) := L^p(0,T; \bb{L}^q)$. Throughout this paper, we denote the scalar product in a Hilbert space $H$ by $\inpro{\cdot}{\cdot}_H$ and its corresponding norm by $\|\cdot\|_H$. We will not distinguish between the scalar product of $\bb{L}^2$ vector-valued functions taking values in $\bb{R}^3$ and the scalar product of $\bb{L}^2$ matrix-valued functions taking values in $\bb{R}^{3\times 3}$, and denote them by $\langle\cdot,\cdot\rangle$.

Finally, the constant $C$ in the estimate denotes a generic constant which takes different values at different occurrences. If
the dependence of $C$ on some variable, e.g.~$T$, is highlighted, we will write
$C(T)$.

\subsection{SAV formulations}\label{sec:sav reform}

Without loss of generality for the analysis, we now set $\sigma = \mu = 1$ in~\eqref{equ:llb a} for simplicity.
Let $\mathcal{E}[\bff{u}]$ be the energy functional given by \eqref{equ:energy}.
Define
\begin{align}\label{equ:F u}
	\mathcal{F}[\bff{u}] := \int_\mathscr{D} \frac{\kappa}{4} \left(\abs{\bff{u}}^4 + 1\right) \,\dx.
\end{align}
Note that $\bff{H}= -\delta \mathcal{E}/\delta \bff{u}$ and that $\mathcal{F}[\bff{u}]\geq \frac14 \kappa \abs{\mathscr{D}}$. We also define the function $g$ by
\begin{align}\label{equ:g u}
	g(\bff{u}):= \kappa \abs{\bff{u}}^2 \bff{u}.
\end{align}
We introduce a scalar variable $r(t):= \sqrt{\mathcal{F}[\bff{u}]}$, so that $r(0)= \sqrt{\mathcal{F}[\bff{u}_0]}$. With these functions, we rewrite the LLB equation \eqref{equ:llb a} by employing the SAV approach as follows:
\begin{subequations}\label{equ:llb sav a}
	\begin{alignat}{2}
		\label{equ:llb sav eq1}
		&\partial_t \bff{u}
		=
		-
		\gamma \bff{u} \times \bff{H}
		+
		\alpha \bff{H}
		\,
		\qquad && \text{for $(t,\bff{x})\in(0,T)\times\mathscr{D}$,}
		\\
		\label{equ:llb sav eq2}
		&\bff{H}= \Delta \bff{u} - \kappa \bff{u} - \frac{r(t)}{\sqrt{\mathcal{F}[\bff{u}]}} g(\bff{u})
		\qquad && \text{for $(t,\bff{x})\in(0,T)\times\mathscr{D}$,}
		\\
		\label{equ:llb sav eq3}
		&\partial_t r= \frac{1}{2\sqrt{\mathcal{F}[\bff{u}]}} \inpro{g(\bff{u})}{\partial_t \bff{u}}
		\qquad && \text{for $t\in(0,T)$,}
		\\
		\label{equ:llb sav init}
		&\bff{u}(0,\bff{x})= \bff{u}_0(\bff{x}) 
		\qquad && \text{for } \bff{x}\in \mathscr{D},
		\\
		\label{equ:llb sav bound}
		&\displaystyle{
			\frac{\partial \bff{u}}{\partial \bff{n}}= \bff{0}}
		\qquad && \text{for } (t,\bff{x})\in (0,T) \times \partial \mathscr{D},
	\end{alignat}
\end{subequations}
At the continuous level, this system is equivalent to \eqref{equ:llb a} as $r(t)/\sqrt{\mathcal{F}[\bff{u}]} \equiv 1$ by definition.
Taking the inner product of \eqref{equ:llb sav eq1} and \eqref{equ:llb sav eq2} with $\bff{H}$ and $\partial_t \bff{u}$, respectively, we obtain
\begin{align}\label{equ:ene law sav form}
	\ddt \mathcal{E}[\bff{u}] \leq 0,	
\end{align}
which is the energy dissipation law obeyed by~\eqref{equ:llb a}.

\subsection{Finite element approximation}

Let $\mathscr{D}\subset \bb{R}^d$, $d=1,2, 3$, be a convex polytopal domain. Let $\{\mathcal{T}_h\}_{h>0}$ be a family of quasi-uniform triangulations of $\mathscr{D}$
into intervals (in 1D), triangles (in 2D), or tetrahedra (in 3D)
with maximal mesh-size $h$.
To discretise the LLB equation, we
introduce the conforming finite element space $\bb{V}_h \subset \bb{H}^1$ given by
\begin{equation}\label{equ:Vh}
	\bb{V}_h := \{\bff{\phi}_h \in \mathcal{C}(\overline{\mathscr{D}}; \bb{R}^3): \bff{\phi}|_K \in \mathcal{P}_1(K;\bb{R}^3), \; \forall K \in \mathcal{T}_h\},
\end{equation}
where $\mathcal{P}_1(K; \bb{R}^3)$ denotes the space of linear polynomials on $K$ taking values in $\bb{R}^3$.

As a consequence of the Bramble--Hilbert lemma, for $p\in [1,\infty]$, there exists a constant $C$ independent of $h$ such that for any $\bff{v} \in \bb{W}^{2,p}$, we have
\begin{align}\label{equ:fin approx}
	\inf_{\chi \in {\bb{V}}_h} \left\{ \norm{\bff{v} - \bff{\chi}}{\bb{L}^p} 
	+ 
	h \norm{\nabla (\bff{v}-\bff{\chi})}{\bb{L}^p} 
	\right\} 
	\leq 
	C h^2 \norm{\bff{v}}{\bb{W}^{2,p}}.
\end{align}
%Moreover, if the triangulation $\mathcal{T}_h$ is quasi-uniform, we have the following inverse estimate:
%\begin{align}\label{equ:inverse}
%	\norm{\bff{\chi}}{\bb{W}^{1,p}} \leq Ch^{-d \left(\frac12-\frac{1}{p}\right)} \norm{\bff{\chi}}{\bb{H}^1}, \quad \forall \bff{\chi}\in \bb{V}_h.
%\end{align}

In the analysis, we shall use several projection and interpolation operators. The existence of such operators and the properties that they possess will be described below (also see~\cite{BreSco08, CroTho87, DouDupWah74}).
Firstly, there exists an orthogonal projection operator $\Pi_h: \bb{L}^2 \to \bb{V}_h$ such that
\begin{align}\label{equ:orth proj}
	\inpro{\Pi_h \bff{v}-\bff{v}}{\bff{\chi}}=0,
	\quad
	\forall \bff{\chi}\in \bb{V}_h,
\end{align}
with the property that for any $\bff{v}\in \bb{W}^{2,p}$,
\begin{align}\label{equ:proj approx}
	\norm{\bff{v}- \Pi_h\bff{v}}{\bb{L}^p}
	+
	h \norm{\nabla( \bff{v}-\Pi_h\bff{v})}{\bb{L}^p}
	\leq
	Ch^2 \norm{\bff{v}}{\bb{W}^{2,p}}.
\end{align}
%Furthermore, there exists a nodal interpolation operator $\mathcal{I}_h:\bb{H}^1 \to \bb{V}_h$ that satisfies
%\begin{align}\label{equ:interp approx}
%	\norm{\bff{v}- \mathcal{I}_h \bff{v}}{\bb{L}^p}
%	+
%	h \norm{\nabla \left(\bff{v}-\mathcal{I}_h \bff{v}\right)}{\bb{L}^p}
%	\leq
%	Ch^{r+1} \norm{\bff{v}}{\bb{W}^{r+1,p}}.
%\end{align}

Next, we introduce the discrete Laplacian operator $\Delta_h: \bb{V}_h \to \bb{V}_h$ defined by
\begin{align}\label{equ:disc laplacian}
	\inpro{\Delta_h \bff{v}_h}{\bff{\chi}}
	=
	- \inpro{\nabla \bff{v}_h}{\nabla \bff{\chi}},
	\quad 
	\forall \bff{v}_h, \bff{\chi} \in \bb{V}_h,
\end{align}
as well as the Ritz projection $R_h: \bb{H}^1 \to \bb{V}_h$ defined by
\begin{align}\label{equ:Ritz}
	\inpro{\nabla R_h \bff{v}- \nabla \bff{v}}{\nabla \bff{\chi}}=0,
	\quad
	\forall \bff{\chi}\in \bb{V}_h.
\end{align}
For any $\bff{v}\in \bb{W}^{2,p}$, the approximation property for the Ritz projection holds \cite{BreSco08, LeyLi21, LinThoWah91, RanSco82}, namely:
\begin{align}\label{equ:Ritz ineq}
	\norm{\bff{v}-R_h\bff{v}}{\bb{W}^{s,p}}
	&\leq
	C h^{2-s} \norm{\bff{v}}{\bb{W}^{2, p}}, \quad s\in \{0,1\},\; p\in (1,\infty).
\end{align}
Furthermore, we have
\begin{align}
	\label{equ:Ritz infty ineq}
	\norm{\bff{v}-R_h\bff{v}}{\bb{L}^\infty}
	&\leq
	Ch^2 \abs{\ln h}^{\frac12} \norm{\bff{v}}{\bb{W}^{2,\infty}},
\end{align}
where $C$ is independent of $h$.

\subsection{Auxiliary results}

{
Auxiliary results relevant to our analysis are collected in this section.
Firstly, in the proof of the error estimates, the following vector identity is often used:
\begin{align}\label{equ:a dot ab}
	2\bff{a}\cdot (\bff{a}-\bff{b})
	=
	\abs{\bff{a}}^2 - \abs{\bff{b}}^2 + \abs{\bff{a}-\bff{b}}^2,\quad \forall \bff{a},\bff{b}\in \bb{R}^3.
\end{align}

Next, we state the Sobolev embedding theorem for $d\in \{1,2,3\}$ and the discrete Gronwall lemma~\cite{Bre11, Soe25}.

\begin{lemma}
	Let
		\begin{equation}\label{equ:p}
			p \in 
			\begin{cases}
				[1,\infty], &\text{ if $d=1$}, \\
				[1,\infty), &\text{ if $d=2$}, \\
				[1,6], &\text{ if $d=3$}.
			\end{cases}
		\end{equation}
		There exists a constant $C:=C(\mathscr{D},p)$ such that
		\begin{align}
			\label{equ:sob embed}
			\norm{\bff{v}}{\bb{L}^p}
			&\leq
			C \norm{\bff{v}}{\bb{H}^1}, \quad \forall \bff{v}\in \bb{H}^1,
			\\
			\label{equ:nab v less Delta v}
			\norm{\nabla \bff{v}}{\bb{L}^p}
			&\leq
			C\norm{\Delta \bff{v}}{\bb{L}^2}, \quad \forall \bff{v}\in \bb{H}^2_{\bff{n}},
		\end{align}
	where $\bb{H}^2_{\bff{n}}$ is the domain of the Neumann Laplacian.
\end{lemma}

\begin{lemma}[Discrete Gronwall's lemma]\label{lem:disc gron}
	Let $k, B, a_j, b_j$, and $\gamma_j$ be non-negative numbers (for all integers $j\geq 0$) such that
	\begin{align*}
		a_n+k \sum_{j=0}^n b_j \leq B + k \sum_{j=0}^{n-1} \gamma_j a_j.
	\end{align*}
	Then
	\begin{align*}
		a_n+k \sum_{j=0}^n b_j \leq B \exp\left({k} \sum_{j=0}^{n-1} \gamma_j\right).
	\end{align*}
\end{lemma}
}

Finally, in two-dimensional polygonal domains, the discrete Sobolev inequality holds~\cite{Bre04}: there exists a constant $C$ independent of $h$ such that for all $v_h\in \bb{V}_h$,
\begin{align}\label{equ:disc sob 2d}
	\norm{\bff{v}_h}{\bb{L}^\infty} \leq C \abs{\ln h}^{\frac12} \norm{\bff{v}_h}{\bb{H}^1}.
\end{align}
Finally, for any $\bff{v}_h\in \bb{V}_h$, the following discrete Gagliardo--Nirenberg inequality holds:
\begin{align}
	\label{equ:disc lapl L infty}
	\norm{\bff{v}_h}{\bb{L}^\infty}
	&\leq
	C \norm{\bff{v}_h}{\bb{L}^2}^{1-\frac{d}{4}} \left(\norm{\bff{v}_h}{\bb{L}^2}^\frac{d}{4} + \norm{\Delta_h \bff{v}_h}{\bb{L}^2}^\frac{d}{4} \right).
\end{align}
Inequality~\eqref{equ:disc lapl L infty} is shown in \cite[Appendix A]{GuiLiWan22}).

%\subsection{Auxiliary results}
%
%We need some auxiliary results in our analysis. 
%
%\begin{lemma}
%Let $\epsilon>0$ be given and $d\in \{1,2,3\}$. Then there exists a positive constant $C$ such that the following inequalities hold.
%\begin{enumerate}
%	\renewcommand{\labelenumi}{\theenumi}
%	\renewcommand{\theenumi}{{\rm (\roman{enumi})}}	
%	\item For any $\bff{v}\in \bb{H}^1$,
%	\begin{align}\label{equ:gal nir uh L4}
	%		\norm{\bff{v}}{\bb{L}^4}
	%		&\leq
	%		C \norm{\bff{v}}{\bb{H}^1}^{\frac{d}{4}}
	%		\norm{\bff{v}}{\bb{L}^2}^{1-\frac{d}{4}}
	%		\leq
	%		C \norm{\bff{v}}{\bb{L}^2}^2
	%		+
	%		\epsilon \norm{\nabla \bff{v}}{\bb{L}^2}^2.
	%	\end{align}
%	\item Let $\mathscr{D}$ be a convex polygonal or polyhedral domain with globally quasi-uniform triangulation. For any $\bff{v}_h\in \bb{V}_h$,
%	\begin{align}
	%		\label{equ:disc lapl L infty}
	%		\norm{\bff{v}_h}{\bb{L}^\infty}
	%		&\leq
	%		C \norm{\bff{v}_h}{\bb{L}^2}^{1-\frac{d}{4}} \left(\norm{\bff{v}_h}{\bb{L}^2}^\frac{d}{4} + \norm{\Delta_h \bff{v}_h}{\bb{L}^2}^\frac{d}{4} \right).
	%	\end{align}
%\end{enumerate}
%\end{lemma}
%
%\begin{proof}
%Estimate \eqref{equ:gal nir uh L4} follows from the Gagliardo--Nirenberg and the Young inequalities. Inequality \eqref{equ:disc lapl L infty} is shown in \cite[Appendix A]{GuiLiWan22}).
%\end{proof}

\section{SAV-FEM with first-order semi-implicit time discretisation}\label{sec:sav fem euler}

Let $\bff{u}_h^n$ be the approximation in $\bb{V}_h$ of $\bff{u}(t_n)$, where $t_n:=nk \in [0,T]$ for $n=0,1,\ldots,N$ and $N= \lfloor T/k \rfloor$. For any function $\bff{v}$, we write $\bff{v}^n:= \bff{v}(t_n)$, and define for any $n\in \bb{N}$,
\begin{align*}
	\dtt \bff{v}^n := \frac1k (\bff{v}^n-\bff{v}^{n-1}).
\end{align*}

A linear fully-discrete SAV FEM with first-order semi-implicit time discretisation can be described as follows.
We start with $\bff{u}_h^0= R_h \bff{u}_0 \in \bb{V}_h$ and $r_h^0= \mathcal{F}(\bff{u}_h^0) \in \bb{R}$. Given $(\bff{u}_h^{n-1}, r_h^{n-1})\in \bb{V}_h \times \bb{R}$, we find $(\bff{u}_h^n, r_h^n)\in \bb{V}_h \times \bb{R}$ satisfying
\begin{subequations}\label{equ:fem euler}
	\begin{alignat}{2}
		\label{equ:fem euler u}
		\inpro{\dtt \bff{u}_h^n}{\bff{\phi}} 
		&=
		-\gamma \inpro{\bff{u}_h^{n-1}\times \bff{H}_h^n}{\bff{\phi}}
		+\alpha \inpro{\bff{H}_h^n}{\bff{\phi}},
		\quad &&\forall \bff{\phi} \in \bb{V}_h,
		\\
		\label{equ:fem euler r}
		\dtt r_h^n
		&=
		\frac{1}{2\sqrt{\mathcal{F}[\bff{u}_h^{n-1}]}} \inpro{g(\bff{u}_h^{n-1})}{\dtt \bff{u}_h^n},
	\end{alignat}
\end{subequations}
where
\begin{align}\label{equ:fem be Hhn}
	\bff{H}_h^n
	&=
	\Delta_h \bff{u}_h^n
	- \kappa \bff{u}_h^n
	- \Pi_h \left[ \frac{r_h^n}{\sqrt{\mathcal{F}[\bff{u}_h^{n-1}]}} g(\bff{u}_h^{n-1}) \right].
\end{align}
Here, $\Delta_h$ and $\Pi_h$ denote the discrete Laplacian and the $\bb{L}^2$ projector, respectively, and the functions $\mathcal{F}$ and $g$ are defined in \eqref{equ:F u} and \eqref{equ:g u}.
The above scheme can be effectively implemented in the following algorithm:

\begin{algorithm}[Semi-implicit linear Euler--SAV FEM]\label{alg:euler}
Let $h>0$ and $k>0$ be given.
\\
\textbf{Input}: Given $\bff{u}_h^0\in \bb{V}_h$ and $r_h^0=\mathcal{F}[\bff{u}_h^0]\in \bb{R}$.
\\
\textbf{For} $n=1$ to $N$, where $N=\lfloor T/k \rfloor$, \textbf{do}:
\begin{enumerate}
	\item Substitute equation \eqref{equ:fem be Hhn} into \eqref{equ:fem euler u}, and compute $\bff{u}_h^n\in \bb{V}_h$;
	\item Compute $r_h^n$ using \eqref{equ:fem euler r}.
\end{enumerate}
\textbf{Output}: a sequence of discrete functions $\{\bff{u}_h^n\}_{1\leq n\leq N}$ and real numbers $\{r_h^n\}_{1\leq n\leq N}$.
\end{algorithm}

Let us define the modified energy $\widetilde{\mathcal{E}}$ of the finite element solution $\bff{u}_h^n$ as
\begin{align}\label{equ:ener modified}
	\widetilde{\mathcal{E}}[\bff{u}_h^n] := \frac12 \norm{\nabla \bff{u}_h^n}{\bb{L}^2}^2 + \frac{\kappa}{2} \norm{\bff{u}_h^n}{\bb{L}^2}^2 + \abs{r_h^n}^2.
\end{align}
Note that $\abs{r_h^n}^2$ is an approximation to $\mathcal{F}[u^n]$, and thus the modified energy $\widetilde{\mathcal{E}}$ approximates the actual energy $\mathcal{E}$.
To derive optimal-order convergence, throughout this section we assume the following regularity for the exact solution $\bff{u}$:
\begin{align}\label{equ:reg u euler}
	\bff{u} \in L^\infty_T(\bb{W}^{2,\infty}) \cap W^{1,\infty}_T(\bb{H}^1) \cap W^{2,\infty}_T(\bb{L}^2).
\end{align}

We first show that the numerical scheme is well-posed.

\begin{proposition}\label{pro:well euler}
	Let $\bff{u}_h^{n-1}\in \bb{V}_h$ and $r_h^{n-1}\in \bb{R}$ be given. For sufficiently small $k>0$, there exists a unique $\bff{u}_h^n\in \bb{V}_h$ and $r_h^n\in \bb{R}$ solving the system \eqref{equ:fem euler}--\eqref{equ:fem be Hhn}.
\end{proposition}

\begin{proof}
For each $h>0$, define a linear operator $S_h:\bb{V}_h\to \bb{V}_h$ by $S_h \bff{v}:= \Pi_h \left(\bff{u}_h^{n-1} \times \bff{v}\right)$. Note that $S_h$ is skew-symmetric with respect to the $\bb{L}^2$-inner product, in the sense that $\inpro{S_h \bff{v}}{\bff{v}}=0$ for all $\bff{v}\in \bb{V}_h$. Consequently, for any $\alpha>0$, the linear operator $T_h:=\alpha I- \gamma S_h$ is positive definite, and thus is invertible on $\bb{V}_h$. We denote its inverse by $T_h^{-1}:= (\alpha I- \gamma S_h)^{-1}$, which is also positive definite. Moreover, $\left\|T_h^{-1}\right\| \leq 1/\alpha$.

Now, equation \eqref{equ:fem euler u} can be written as $\bff{H}_h^n= \frac{1}{k} T_h^{-1} \left(\bff{u}_h^n-\bff{u}_h^{n-1}\right)$, which by \eqref{equ:fem be Hhn} implies
\begin{align*}
	T_h^{-1} \left(\bff{u}_h^n-\bff{u}_h^{n-1}\right) 
	&=
	k\Delta_h \bff{u}_h^n- k\kappa \bff{u}_h^n
	-
	\frac{k r_h^n}{\sqrt{\mathcal{F}[\bff{u}_h^{n-1}]}} \Pi_h\left[g(\bff{u}_h^{n-1})\right].
\end{align*}
Let $\mathcal{A}_h:=-k\Delta_h +k\kappa I+ T_h^{-1}$.
Using \eqref{equ:fem euler r} and rearranging the above equation, we see that \eqref{equ:fem euler}--\eqref{equ:fem be Hhn} is equivalent to finding $\bff{u}_h^n \in \bb{V}_h$ which satisfies
\begin{align}\label{equ:lin sys}
	\mathcal{A}_h \bff{u}_h^n + \frac{k}{2\sqrt{\mathcal{F}[\bff{u}_h^{n-1}]}} \inpro{g(\bff{u}_h^{n-1})}{\bff{u}_h^n} \Pi_h\left[g(\bff{u}_h^{n-1})\right]
	=
	\bff{f}_h^{n-1},
\end{align}
where $\bff{f}_h^{n-1}$ contains only data from the previous step, namely
\begin{align*}
	\bff{f}_h^{n-1}:= 
	T_h^{-1} \bff{u}_h^{n-1} 
	-
	\frac{k r_h^{n-1}}{\sqrt{\mathcal{F}[\bff{u}_h^{n-1}]}} \Pi_h\left[g(\bff{u}_h^{n-1})\right]
	+
	\frac{k}{2\sqrt{\mathcal{F}[\bff{u}_h^{n-1}]}}
	\inpro{g(\bff{u}_h^{n-1})}{\bff{u}_h^{n-1}}
	\Pi_h\left[g(\bff{u}_h^{n-1})\right].
\end{align*}
Since $\bb{V}_h$ is finite-dimensional, it suffices to show that the linear system \eqref{equ:lin sys} has at most one solution; uniqueness then implies existence. Writing \eqref{equ:lin sys} as $\mathcal{B}_h \bff{u}_h^n=\bff{f}_h^{n-1}$, we note that the operator $\mathcal{B}_h$ is positive definite on $\bb{V}_h$. The desired uniqueness then follows immediately, completing the proof of the proposition.
\end{proof}

Next, we show the unconditional stability property of scheme~\eqref{equ:fem euler} by establishing a modified energy law satisfied by the discrete solution.

\begin{proposition}
	Let $\bff{u}_h^n$ be given by Algorithm~\ref{alg:euler}. For $n=1,2,\ldots,\lfloor T/k \rfloor$,
	\begin{align}\label{equ:disc ener ineq}
		\widetilde{\mathcal{E}}[\bff{u}_h^n] \leq \widetilde{\mathcal{E}}[\bff{u}_h^{n-1}],
	\end{align}
	where $\widetilde{\mathcal{E}}[\bff{u}_h^n]$ was defined in \eqref{equ:ener modified}.
\end{proposition}

\begin{proof}
	Setting $\bff{\phi}=\bff{H}_h^n$ in \eqref{equ:fem euler} and taking the inner product of \eqref{equ:fem be Hhn} with $\dtt \bff{u}_h^n$, then combining the resulting equations, we obtain
	\begin{align}\label{equ:ener est equal}
		&\frac1k \left(\widetilde{\mathcal{E}}[\bff{u}_h^n]- \widetilde{\mathcal{E}}[\bff{u}_h^{n-1}]\right)
		+
		\alpha \norm{\bff{H}_h^n}{\bb{L}^2}^2
		\nonumber\\
		&\quad
		+
		\frac1k \left(\frac12 \norm{\nabla \bff{u}_h^n-\nabla \bff{u}_h^{n-1}}{\bb{L}^2}^2 + \frac{\kappa}{2} \norm{\bff{u}_h^n-\bff{u}_h^{n-1}}{\bb{L}^2}^2 
		+ \abs{r_h^n-r_h^{n-1}}^2 \right) = 0.
	\end{align}
	This implies \eqref{equ:disc ener ineq}, completing the proof of the proposition.
\end{proof}

With the energy stability property, we can show the following uniform boundedness of the finite element solutions.

\begin{lemma}
Let $\bff{u}_h^n$ be given by Algorithm~\ref{alg:euler}. For $n=1,2,\ldots,\lfloor T/k \rfloor$, we have
\begin{align}\label{equ:uhn H1 stab}
	\norm{\bff{u}_h^n}{\bb{H}^1}^2 + \abs{r_h^n}^2 
	+ k \sum_{j=1}^n \norm{\bff{H}_h^j}{\bb{L}^2}^2 \leq C_\mathrm{S}.
\end{align}
Consequently, we have
\begin{align}\label{equ:C infty euler}
	k \sum_{j=1}^n \norm{\bff{u}_h^j}{\bb{L}^\infty}^2 \leq C_\infty,
\end{align}
where $C_\mathrm{S}$ and $C_\infty$ are constants depending on $T$, but are independent of $n$, $h$, and $k$.
\end{lemma}

\begin{proof}
By summing \eqref{equ:ener est equal} over $j\in \{1,2,\ldots, n\}$, we have
\begin{align*}
	\widetilde{\mathcal{E}}[\bff{u}_h^n]+ \alpha k\sum_{j=1}^n \norm{\bff{H}_h^j}{\bb{L}^2}^2 \leq \widetilde{\mathcal{E}}[\bff{u}_h^0],
\end{align*}
which implies \eqref{equ:uhn H1 stab}.
By rearranging the terms, noting the stability of $\Pi_h$, we obtain
	\begin{align*}
		k\sum_{j=1}^n \norm{\Delta_h \bff{u}_h^j}{\bb{L}^2}^2
		&\leq
		Ck\sum_{j=1}^n \norm{\bff{H}_h^j}{\bb{L}^2}^2
		+
		Ck \sum_{j=1}^n \norm{\bff{u}_h^j}{\bb{L}^2}^2
		+
		Ck \sum_{j=1}^n \abs{r_h^j}^2 \norm{g(\bff{u}_h^{j-1})}{\bb{L}^2}^2
		\\
		&\leq
		Ck\sum_{j=1}^n \norm{\bff{H}_h^j}{\bb{L}^2}^2
		+
		CTC_\mathrm{S}
		+
		Ck\sum_{j=1}^n \abs{r_h^j}^2 \norm{\bff{u}_h^j}{\bb{L}^6}^6
		\\
		&\leq
		Ck\sum_{j=1}^n \norm{\bff{H}_h^j}{\bb{L}^2}^2
		+
		CTC_\mathrm{S}
		+
		CTC_{\mathrm{S}}^4 \leq C,
	\end{align*}
	where $C_\mathrm{S}$ is the constant in \eqref{equ:uhn H1 stab}. Inequality \eqref{equ:C infty euler} then follows from \eqref{equ:disc lapl L infty}, noting \eqref{equ:uhn H1 stab} again.
\end{proof}

To facilitate the error analysis of the scheme, we decompose the approximation error by writing
\begin{align}
	\label{equ:split un}
	\bff{u}_h^n-\bff{u}^n &= (\bff{u}_h^n -R_h \bff{u}^n) + (R_h \bff{u}^n- \bff{u}^n) =: \bff{\theta}^n + \bff{\rho}^n,
	\\
	\label{equ:split Hn}
	\bff{H}_h^n-\bff{H}^n &= (\bff{H}_h^n -R_h \bff{H}^n) + (R_h \bff{H}^n- \bff{H}^n) =: \bff{\xi}^n + \bff{\eta}^n,
	\\
	\label{equ:split rn}
	r_h^n-r^n &= e^n,
\end{align}
where $R_h$ is the Ritz projection defined by \eqref{equ:Ritz}. As such,
\begin{align}\label{equ:Ritz zero}
	\inpro{\nabla \bff{\rho}^n}{\nabla \bff{\phi}} = \inpro{\nabla \bff{\eta}^n}{\nabla \bff{\phi}} = 0, \quad \forall \bff{\phi}\in \bb{V}_h.
\end{align}
We begin with some preliminary results. Firstly, note that for $p\in (1,\infty)$,
\begin{align}\label{equ:dt vn Lp}
	\norm{\dtt \bff{v}^n}{\bb{L}^p}
	&=
	\norm{\frac{1}{k} \int_{t_{n-1}}^{t_n} \partial_t \bff{v}(t) \,\dt}{\bb{L}^p}
	\leq
	\norm{\bff{v}}{W^{1,\infty}_T(\bb{L}^p)},
	\\
	\label{equ:dt vn min par v}
	\norm{\dtt \bff{v}^n- \partial_t \bff{v}^n}{\bb{L}^p}
	&=
	\norm{\frac{1}{2k} \int_{t_{n-1}}^{t_n} (t-t_{n-1}) \partial_t^2 \bff{v}(t)\,\dt}{\bb{L}^p}
	\leq 
	C_p k \norm{\bff{v}}{W^{2,\infty}_T(\bb{L}^p)}.
\end{align}

The following lemmas will be needed in the error analysis.

\begin{lemma}\label{lem:nonlinear euler}
	Let $\bff{u}_h^n$ be given by Algorithm~\ref{alg:euler}, and let $\bff{\theta}^n$ and $\bff{\rho}^n$ be given by \eqref{equ:split un}. Let the functions $\mathcal{F}$ and $g$ be given by \eqref{equ:F u} and \eqref{equ:g u}, respectively. Then
	\begin{align}
		\label{equ:g uhn F}
		\norm{\frac{g(\bff{u}_h^{n-1})}{\sqrt{\mathcal{F}[\bff{u}_h^{n-1}]}} - \frac{g(\bff{u}^{n-1})}{\sqrt{\mathcal{F}[\bff{u}^{n-1}]}}}{\bb{L}^2}
		&\leq
		C \norm{\bff{\theta}^{n-1}+ \bff{\rho}^{n-1}}{\bb{L}^6},
		\\[1ex]
		\label{equ:g un F}
		\norm{\frac{g(\bff{u}^n)}{\sqrt{\mathcal{F}[\bff{u}^n]}} - \frac{ g(\bff{u}^{n-1})}{\sqrt{\mathcal{F}[\bff{u}^{n-1}]}}}{\bb{L}^2}
		&\leq
		Ck,
	\end{align}
	where $C$ is a constant {depending on $T$ and the exact solution $\bff{u}$}, but is independent of $n$, $h$, and $k$.
\end{lemma}

\begin{proof}
	We first prove \eqref{equ:g uhn F}. Note that the left-hand side of \eqref{equ:g uhn F} can be written as
	\begin{align}\label{equ:est gf}
		&\norm{\frac{\big(g(\bff{u}_h^{n-1})-g(\bff{u}^{n-1})\big)\sqrt{\F{\bff{u}^{n-1}}}}{\sqrt{\F{\bff{u}_h^{n-1}}} \cdot \sqrt{\F{\bff{u}^{n-1}}}} + \frac{g(\bff{u}^{n-1}) \big(\F{\bff{u}^{n-1}}- \F{\bff{u}_h^{n-1}}\big)}{\sqrt{\F{\bff{u}_h^{n-1}}} \cdot \sqrt{\F{\bff{u}^{n-1}}} \cdot \left(\sqrt{\F{\bff{u}^{n-1}}}+ \sqrt{\F{\bff{u}_h^{n-1}}}\right) }}{\bb{L}^2}
		\nonumber\\
		&\qquad\leq
		C\norm{g(\bff{u}_h^{n-1})-g(\bff{u}^{n-1})}{\bb{L}^2} + C\abs{\F{\bff{u}^{n-1}}- \F{\bff{u}_h^{n-1}}}.
	\end{align}
	By writing
	\begin{align*}
		g(\bff{u}_h^{n-1})-g(\bff{u}^{n-1})
		=
		\left( (\bff{\theta}^{n-1}+\bff{\rho}^{n-1})\cdot (\bff{u}_h^{n-1}+\bff{u}^{n-1}) \right) \bff{u}_h^{n-1} + \abs{\bff{u}^{n-1}}^2 \left(\bff{\theta}^{n-1}+\bff{\rho}^{n-1}\right),
	\end{align*}
	we obtain by H\"older's inequality, \eqref{equ:uhn H1 stab}, and the Sobolev embedding $\bb{H}^1\hookrightarrow \bb{L}^6$,
	\begin{align}\label{equ:guhn int}
		\norm{g(\bff{u}_h^{n-1})-g(\bff{u}^{n-1})}{\bb{L}^2}
		&\leq
		C\left(\norm{\bff{u}_h^{n-1}}{\bb{L}^6}^2 + \norm{\bff{u}^{n-1}}{\bb{L}^6}^2 \right) \norm{\bff{\theta}^{n-1}+ \bff{\rho}^{n-1}}{\bb{L}^6}
		\nonumber\\
		&\leq
		C\norm{\bff{\theta}^{n-1}+ \bff{\rho}^{n-1}}{\bb{L}^6}.
	\end{align}
	Similarly, by writing
	\begin{align*}
		\F{\bff{u}^{n-1}}- \F{\bff{u}_h^{n-1}}
		=
		\int_{\mathscr{D}} \left(\abs{\bff{u}^{n-1}}^2+ \abs{\bff{u}_h^{n-1}}^2\right)\left((\bff{\theta}^{n-1}+\bff{\rho}^{n-1})\cdot (\bff{u}^{n-1}+\bff{u}_h^{n-1})\right) \dx,
	\end{align*}
	we have by H\"older's inequality, \eqref{equ:uhn H1 stab}, and the regularity of $\bff{u}$ that
	\begin{align}\label{equ:F un1 inter}
		&\abs{\F{\bff{u}^{n-1}}- \F{\bff{u}_h^{n-1}}}
		\nonumber\\
		&\leq
		C\left(\norm{\bff{u}_h^{n-1}}{\bb{L}^4}^2 + \norm{\bff{u}^{n-1}}{\bb{L}^4}^2 \right) \left(\norm{\bff{u}^{n-1}}{\bb{L}^4}+\norm{\bff{u}_h^{n-1}}{\bb{L}^4}\right) \norm{\bff{\theta}^{n-1}+ \bff{\rho}^{n-1}}{\bb{L}^4}
		\nonumber\\
		&\leq
		C \norm{\bff{\theta}^{n-1}+ \bff{\rho}^{n-1}}{\bb{L}^4}.
	\end{align}
	Substituting these estimates into \eqref{equ:est gf}, we infer \eqref{equ:g uhn F}.
	
	Next, we show \eqref{equ:g un F}. By applying the same argument, we obtain
	\begin{align*}
		\norm{\frac{g(\bff{u}^n)}{\sqrt{\mathcal{F}[\bff{u}^n]}} - \frac{g(\bff{u}^{n-1})}{\sqrt{\mathcal{F}[\bff{u}^{n-1}]}}}{\bb{L}^2}
		&\leq
		C \norm{\bff{u}^n-\bff{u}^{n-1}}{\bb{L}^6}
		\leq
		Ck \norm{\bff{u}}{W^{1,\infty}_T(\bb{H}^1)},
	\end{align*}
	where in the last step we used \eqref{equ:dt vn Lp}. Noting the regularity of $\bff{u}$, we deduce \eqref{equ:g un F}, thus completing the proof of the lemma.
\end{proof}

We derive a bound for $\norm{\dtt \bff{\theta}^n}{\bb{L}^2}$ in the following lemma.

\begin{lemma}
	Let $\bff{u}_h^n$ be given by Algorithm~\ref{alg:euler}. Let $\bff{\theta}^n$ and $\bff{\rho}^n$ be given by \eqref{equ:split un}, and let $\bff{\xi}^n$ and $\bff{\eta}^n$ be given by \eqref{equ:split Hn}. Then
	\begin{align}\label{equ:dt theta L2}
		\norm{\dtt \bff{\theta}^n}{\bb{L}^2}
		&\leq
		C\left(1+\norm{\bff{u}_h^{n-1}}{\bb{L}^\infty}\right) \norm{\bff{\xi}^n}{\bb{L}^2}
		+
		C\norm{\bff{\theta}^{n-1}}{\bb{L}^2}
		+
		C\left(1+\norm{\bff{u}_h^{n-1}}{\bb{L}^\infty}\right) h^2 + Ck,
		%\\
		%\label{equ:dt theta L32}
		%\norm{\dtt \bff{\theta}^n}{\bb{L}^{3/2}}
		%&\leq
		%Ck + Ch^2 + C \norm{\bff{\xi}^n}{\bb{L}^2}
		%+
		%C\norm{\bff{\theta}^{n-1}}{\bb{L}^2}.
	\end{align}
	Consequently, we also have the bound
	\begin{align}\label{equ:dt uh L2}
		\norm{\dtt \bff{u}_h^n}{\bb{L}^2}
		&\leq
		C\left(1+\norm{\bff{u}_h^{n-1}}{\bb{L}^\infty}\right) \norm{\bff{\xi}^n}{\bb{L}^2}
		+
		C\norm{\bff{\theta}^{n-1}}{\bb{L}^2}
		+
		C\left(1+\norm{\bff{u}_h^{n-1}}{\bb{L}^\infty}\right) h^2 
		\nonumber\\
		&\quad
		+ Ck
		+ \norm{u}{W^{1,\infty}_T(\bb{L}^2)},
	\end{align}
	where $C$ is a constant {depending on $T$ and the exact solution $\bff{u}$}, but is independent of $n$, $h$, and $k$.
\end{lemma}

\begin{proof}
	Subtracting the weak formulation of \eqref{equ:llb sav eq1} from the corresponding equation in scheme \eqref{equ:fem euler}, we obtain by noting \eqref{equ:split un},
	\begin{align}
		\label{equ:uhn min un}
		&\inpro{\dtt \bff{\theta}^n + \dtt \bff{\rho}^n + \dtt \bff{u}^n - \partial_t \bff{u}^n}{\bff{\phi}}
		\nonumber\\
		&=
		-\gamma \inpro{\bff{u}_h^{n-1} \times (\bff{\xi}^n+\bff{\eta}^n)}{\bff{\phi}}
		-
		\gamma \inpro{(\bff{\theta}^{n-1}+\bff{\rho}^{n-1})\times \bff{H}^n}{\bff{\phi}}
		\nonumber\\
		&\quad
		-
		\gamma \inpro{(\bff{u}^{n-1}-\bff{u}^n)\times \bff{H}^n}{\bff{\phi}}
		+
		\alpha \inpro{\bff{\xi}^n+\bff{\eta}^n}{\bff{\phi}}, \quad \forall \bff{\phi}\in \bb{V}_h.
	\end{align} 
	Therefore, with $\Pi_h$ defined by \eqref{equ:orth proj}, we can write
	\begin{align*}
		\dtt \bff{\theta}^n
		&=
		\Pi_h \big[\partial_t \bff{u}^n- \dtt \bff{u}^n\big] 
		-
		\Pi_h \dtt \bff{\rho}^n
		-
		\gamma \Pi_h \big[\bff{u}_h^{n-1}\times (\bff{\xi}^n+\bff{\eta}^n)\big]
		\\
		&\quad
		-
		\gamma \Pi_h\big[(\bff{\theta}^{n-1}+\bff{\rho}^{n-1})\times \bff{H}^n\big] 
		-
		\gamma \Pi_h \big[(\bff{u}^{n-1}-\bff{u}^n)\times \bff{H}^n\big] 
		+
		\alpha \bff{\xi}^n
		+
		\alpha \Pi_h \bff{\eta}^n.
	\end{align*}
	Using the $\bb{L}^2$ stability of $\Pi_h$, we then have by H\"older's inequality,
	\begin{align*}
		\norm{\dtt \bff{\theta}^n}{\bb{L}^2}
		&\leq
		C\norm{\partial_t \bff{u}^n- \dtt \bff{u}^n}{\bb{L}^2}
		+
		C\norm{\dtt \bff{\rho}^n}{\bb{L}^2}
		+
		C\norm{\bff{u}_h^{n-1}}{\bb{L}^\infty} \norm{\bff{\xi}^n+\bff{\eta}^n}{\bb{L}^2}
		\\
		&\quad
		+
		C\norm{\bff{\theta}^{n-1}+\bff{\rho}^{n-1}}{\bb{L}^2} \norm{\bff{H}^n}{\bb{L}^\infty}
		+
		C\norm{\bff{u}^{n-1}-\bff{u}^n}{\bb{L}^2} \norm{\bff{H}^n}{\bb{L}^\infty}
		+
		C\norm{\bff{\xi}^n}{\bb{L}^2}
		+
		C\norm{\bff{\eta}^n}{\bb{L}^2}
		\\
		&\leq
		Ck + Ch^2 + C \norm{\bff{u}_h^{n-1}}{\bb{L}^\infty} \norm{\bff{\xi}^n}{\bb{L}^2}
		+
		Ch^2 \norm{\bff{u}_h^{n-1}}{\bb{L}^\infty}
		+
		C\norm{\bff{\theta}^{n-1}}{\bb{L}^2}
		+
		C\norm{\bff{\xi}^n}{\bb{L}^2},
	\end{align*}
	where in the last step we used \eqref{equ:dt vn Lp}, \eqref{equ:dt vn min par v}, \eqref{equ:Ritz ineq}, and the regularity of the solution assumed in \eqref{equ:reg u euler}. This implies \eqref{equ:dt theta L2}.
	
%	Similarly, we also have by H\"older's inequality and the stability estimate \eqref{equ:uhn H1 stab},
%	\begin{align*}
%		\norm{\dtt \bff{\theta}^n}{\bb{L}^{3/2}}
%		&\leq
%		C\norm{\partial_t \bff{u}^n- \dtt \bff{u}^n}{\bb{L}^{3/2}}
%		+
%		C\norm{\dtt \bff{\rho}^n}{\bb{L}^{3/2}}
%		+
%		C\norm{\bff{u}_h^{n-1}}{\bb{L}^6} \norm{\bff{\xi}^n+\bff{\eta}^n}{\bb{L}^2}
%		\\
%		&\quad
%		+
%		C\norm{\bff{\theta}^{n-1}+\bff{\rho}^{n-1}}{\bb{L}^2} \norm{\bff{H}^n}{\bb{L}^6}
%		+
%		C\norm{\bff{u}^{n-1}-\bff{u}^n}{\bb{L}^2} \norm{\bff{H}^n}{\bb{L}^6}
%		+
%		C\norm{\bff{\xi}^n}{\bb{L}^{3/2}}
%		+
%		C\norm{\bff{\eta}^n}{\bb{L}^{3/2}}
%		\\
%		&\leq
%		Ck + Ch^2 + C \norm{\bff{\xi}^n}{\bb{L}^2}
%		+
%		C\norm{\bff{\theta}^{n-1}}{\bb{L}^2}.
%	\end{align*}
	
	Finally, note that $\dtt \bff{u}_h^n= \dtt \bff{\theta}^n+ \dtt \bff{\rho}^n + \dtt \bff{u}^n$. Therefore, by the triangle inequality, \eqref{equ:dt theta L2}, \eqref{equ:Ritz ineq}, and \eqref{equ:dt vn Lp}, we deduce \eqref{equ:dt uh L2}.
\end{proof}

We can now prove the following key superconvergence estimate.

\begin{proposition}
	Let $\bff{\theta}^n, \bff{\xi}^n$, and $e^n$ be defined by \eqref{equ:split un}, \eqref{equ:split Hn}, and \eqref{equ:split rn}, respectively. For sufficiently small $h,k>0$, we have 
	\begin{align}\label{equ:error theta n eul}
		\norm{\bff{\theta}^n}{\bb{H}^1}^2 + \abs{e^n}^2 + \sum_{j=1}^n \norm{\bff{\xi}^j}{\bb{L}^2}^2 \leq C(h^4+k^2).
	\end{align}
	Furthermore, if $\mathscr{D}\subset \bb{R}^2$, then
	\begin{align}\label{equ:error theta Linfty eul}
		\norm{\bff{\theta}^n}{\bb{L}^\infty}^2
		\leq
		C(h^4+k^2)\abs{\ln h}.
	\end{align}
	Here, $C$ is a constant {depending on $T$ and the exact solution $\bff{u}$}, but is independent of $n$, $h$, and $k$.
\end{proposition}

\begin{proof}
	Firstly, recall that by subtracting the weak formulation of \eqref{equ:llb sav eq1} at time $t=t_n$ from the corresponding equation in scheme \eqref{equ:fem euler}, we have \eqref{equ:uhn min un}.
	Similarly, noting \eqref{equ:Ritz zero}, we have
	\begin{align}
		\label{equ:Hhn min Hn}
		\inpro{\bff{\xi}^n+\bff{\eta}^n}{\bff{\chi}}
		&=
		- \inpro{\nabla \bff{\theta}^n}{\nabla \bff{\chi}}
		-
		\kappa \inpro{\bff{\theta}^n+\bff{\rho}^n}{\bff{\chi}}
		\nonumber\\
		&\quad
		-
		\frac{e^n}{\sqrt{\mathcal{F}[\bff{u}^{n-1}]}} \inpro{g(\bff{u}^{n-1})}{\bff{\chi}}
		-
		r_h^n \inpro{\frac{g(\bff{u}_h^{n-1})}{{\mathcal{F}[\bff{u}_h^{n-1}]}}- \frac{g(\bff{u}^{n-1})}{{\mathcal{F}[\bff{u}^{n-1}]}}}{\bff{\chi}}
		\nonumber\\
		&\quad
		+
		r^n \inpro{\frac{g(\bff{u}^n)}{\sqrt{\mathcal{F}[\bff{u}^n]}} - \frac{g(\bff{u}^{n-1})}{\sqrt{\mathcal{F}[\bff{u}^{n-1}]}}}{\bff{\chi}}, \quad \forall \bff{\chi}\in \bb{V}_h.
	\end{align}
	Furthermore, subtracting \eqref{equ:llb sav eq3} from the second equation in \eqref{equ:fem euler}, we obtain
	\begin{align}\label{equ:rhn min rn}
		&\dtt e^n + \dtt r^n - \partial_t r^n 
		\nonumber\\
		&=
		\frac{1}{2\sqrt{\mathcal{F}[\bff{u}_h^{n-1}]}} \inpro{g(\bff{u}_h^{n-1})}{\dtt \bff{\theta}^n+ \dtt \bff{\rho}^n}
		+
		\inpro{\frac{g(\bff{u}_h^{n-1})}{2\sqrt{\mathcal{F}[\bff{u}_h^{n-1}]}} - \frac{g(\bff{u}^{n-1})}{2\sqrt{\mathcal{F}[\bff{u}^{n-1}]}}}{\dtt \bff{u}^n}
		\nonumber\\
		&\quad
		+
		\inpro{\frac{g(\bff{u}^{n-1})}{2\sqrt{\mathcal{F}[\bff{u}^{n-1}]}}}{\dtt \bff{u}^n-\partial_t \bff{u}^n}.
	\end{align}
	Now, noting \eqref{equ:a dot ab}, we set $\bff{\phi}= \bff{\xi}^n$ in \eqref{equ:uhn min un} and $\bff{\chi}= \dtt \bff{\theta}^n$ in \eqref{equ:Hhn min Hn}, and multiply \eqref{equ:rhn min rn} by $2e^n$. We then add the resulting equations to obtain
	\begin{align}\label{equ:ineq theta e}
		&\frac{1}{2k} \left(\norm{\nabla \bff{\theta}^n}{\bb{L}^2}^2 - \norm{\nabla \bff{\theta}^{n-1}}{\bb{L}^2}^2 \right) 
		+
		\frac{\kappa}{2k} \left(\norm{\bff{\theta}^n}{\bb{L}^2}^2 - \norm{\bff{\theta}^{n-1}}{\bb{L}^2}^2 \right)
		+
		\frac{1}{k}  \left(\abs{e^n}^2- \abs{e^{n-1}}^2\right)
		\nonumber\\
		&\quad
		+
		\frac{1}{2k} \left(\norm{\nabla \bff{\theta}^n-\nabla \bff{\theta}^{n-1}}{\bb{L}^2}^2 + \kappa \norm{\bff{\theta}^n-\bff{\theta}^{n-1}}{\bb{L}^2}^2\right) 
		+
		\frac{1}{k} \abs{e^n-e^{n-1}}^2
		+
		\alpha \norm{\bff{\xi}^n}{\bb{L}^2}^2
		\nonumber\\
		&=
		\inpro{\dtt \bff{\rho}^n+\dtt \bff{u}^n-\partial_t \bff{u}^n}{\bff{\xi}^n}
		-
		\inpro{\bff{\eta}^n}{\dtt \bff{\theta}^n}
		+
		\gamma \inpro{\bff{u}_h^{n-1}\times \bff{\eta}^n}{\bff{\xi}^n}
		\nonumber\\
		&\quad
		+
		\gamma \inpro{(\bff{\theta}^{n-1}+\bff{\rho}^{n-1})\times \bff{H}^n}{\bff{\xi}^n}
		+
		\gamma \inpro{(\bff{u}^{n-1}-\bff{u}^n)\times \bff{H}^n}{\bff{\xi}^n}
		-
		\alpha \inpro{\bff{\eta}^n}{\bff{\xi}^n}
		-
		\kappa \inpro{\bff{\rho}^n}{\dtt \bff{\theta}^n}
		\nonumber\\
		&\quad
		-
		r_h^n \inpro{\frac{g(\bff{u}_h^{n-1})}{\sqrt{\mathcal{F}[\bff{u}_h^{n-1}]}}- \frac{g(\bff{u}^{n-1})}{\sqrt{\mathcal{F}[\bff{u}^{n-1}]}}}{\dtt \bff{\theta}^n}
		+
		r^n \inpro{\frac{g(\bff{u}^n)}{\sqrt{\mathcal{F}[\bff{u}^n]}} - \frac{g(\bff{u}^{n-1})}{\sqrt{\mathcal{F}[\bff{u}^{n-1}]}}}{\dtt \bff{\theta}^n}
		\nonumber\\
		&\quad
		+
		e^n \inpro{\frac{g(\bff{u}_h^{n-1})}{\sqrt{\mathcal{F}[\bff{u}_h^{n-1}]}}- \frac{g(\bff{u}^{n-1})}{\sqrt{\mathcal{F}[\bff{u}^{n-1}]}}}{\dtt \bff{\theta}^n}
		+
		\frac{e^n}{\sqrt{\mathcal{F}[\bff{u}_h^{n-1}]}} \inpro{g(\bff{u}_h^{n-1})}{\dtt \bff{\rho}^n}
		\nonumber\\
		&\quad
		+
		e^n
		\inpro{\frac{g(\bff{u}_h^{n-1})}{\sqrt{\mathcal{F}[\bff{u}_h^{n-1}]}} - \frac{g(\bff{u}^{n-1})}{\sqrt{\mathcal{F}[\bff{u}^{n-1}]}}}{\dtt \bff{u}^n}
		+
		e^n \inpro{\frac{g(\bff{u}^{n-1})}{\sqrt{\mathcal{F}[\bff{u}^{n-1}]}}}{\dtt \bff{u}^n-\partial_t \bff{u}^n}
		\nonumber\\
		&=: I_1+I_2+\ldots+I_{13}.
	\end{align}
	We will estimate each term $I_j, j=1,2,\ldots,13$, appearing on the right-hand side of \eqref{equ:ineq theta e}. Firstly, by \eqref{equ:dt vn Lp}, \eqref{equ:Ritz ineq}, \eqref{equ:dt vn min par v}, and Young's inequality, we have
	\begin{align*}
		\abs{I_1}\leq Ch^4 + Ck^2 + \frac{\alpha}{20} \norm{\bff{\xi}^n}{\bb{L}^2}^2.
	\end{align*}
	Secondly, by \eqref{equ:Ritz ineq}, \eqref{equ:dt theta L2}, and Young's inequality,
	\begin{align*}
		\abs{I_2}
		\leq
		C\left(1+\norm{\bff{u}_h^{n-1}}{\bb{L}^\infty}^2\right) h^4 
		+
		C\norm{\bff{\theta}^{n-1}}{\bb{L}^2}^2 
		+
		Ck^2
		+
		\frac{\alpha}{20} \norm{\bff{\xi}^n}{\bb{L}^2}^2.
	\end{align*}
	For the third term, by Young's inequality and \eqref{equ:Ritz ineq}, we obtain
	\begin{align*}
		\abs{I_3}
		&\leq
		C\norm{\bff{u}_h^{n-1}}{\bb{L}^\infty}^2 h^4
		+
		\frac{\alpha}{20} \norm{\bff{\xi}^n}{\bb{L}^2}^2.
	\end{align*}
	For the terms $I_4$, $I_5$, $I_6$, and $I_7$, by similar argument we obtain
	\begin{align*}
		\abs{I_4}
		&\leq
		C\norm{\bff{\theta}^{n-1}+ \bff{\rho}^{n-1}}{\bb{L}^2} \norm{\bff{H}^n}{\bb{L}^\infty} \norm{\bff{\xi}^n}{\bb{L}^2}
		\leq
		C\norm{\bff{\theta}^{n-1}}{\bb{L}^2}^2
		+
		Ch^4 
		+
		\frac{\alpha}{20} \norm{\bff{\xi}^n}{\bb{L}^2}^2,
		\\
		\abs{I_5}
		&\leq
		C\norm{\bff{u}^{n-1}-\bff{u}^n}{\bb{L}^2} \norm{\bff{H}^n}{\bb{L}^\infty} \norm{\bff{\xi}^n}{\bb{L}^2}
		\leq
		Ck^2 + 
		\frac{\alpha}{20} \norm{\bff{\xi}^n}{\bb{L}^2}^2,
		\\
		\abs{I_6}
		&\leq
		Ch^4 +
		\frac{\alpha}{20} \norm{\bff{\xi}^n}{\bb{L}^2}^2,
		\\
		\abs{I_7}
		&\leq
		C \left(1+\norm{\bff{u}_h^{n-1}}{\bb{L}^\infty}^2\right) h^4 + Ck^2 + C\norm{\bff{\theta}^{n-1}}{\bb{L}^2}^2 + \frac{\alpha}{20} \norm{\bff{\xi}^n}{\bb{L}^2}^2.
	\end{align*}
	Next, for the term $I_8$, we apply \eqref{equ:uhn H1 stab}, \eqref{equ:g uhn F}, \eqref{equ:dt theta L2}, H\"older's and Young's inequalities, as well as the Sobolev embedding $\bb{H}^1\hookrightarrow \bb{L}^6$ to obtain
	\begin{align*}
		\abs{I_8}
		&\leq
		C \abs{r_h^n} \norm{\frac{g(\bff{u}_h^{n-1})}{\sqrt{\mathcal{F}[\bff{u}_h^{n-1}]}} - \frac{g(\bff{u}^{n-1})}{\sqrt{\mathcal{F}[\bff{u}^{n-1}]}}}{\bb{L}^2} \norm{\dtt \bff{\theta}^n}{\bb{L}^2}
		\\
		&\leq
		C\left(1+\norm{\bff{u}_h^{n-1}}{\bb{L}^\infty}^2 \right) \norm{\bff{\theta}^{n-1}}{\bb{H}^1}^2
		+
		C\left(1+\norm{\bff{u}_h^{n-1}}{\bb{L}^\infty}^2 \right) h^4
		+
		Ck^2 
		+
		\frac{\alpha}{20} \norm{\bff{\xi}^n}{\bb{L}^2}^2.
	\end{align*}
	Similarly, for the next term, by \eqref{equ:g un F} and \eqref{equ:dt theta L2},
	\begin{align*}
		\abs{I_9}
		&\leq
		C\left(1+\norm{\bff{u}_h^{n-1}}{\bb{L}^\infty}^2 \right) k^2 
		+
		C\norm{\bff{\theta}^{n-1}}{\bb{L}^2}^2
		+
		Ch^4.
	\end{align*}
	For the term $I_{10}$, by the same argument as in the estimate for $I_8$, we have
	\begin{align*}
		\abs{I_{10}}
		&\leq
		C\left(1+\norm{\bff{u}_h^{n-1}}{\bb{L}^\infty}^2 \right) \norm{\bff{\theta}^{n-1}}{\bb{H}^1}^2
		+
		C\left(1+\norm{\bff{u}_h^{n-1}}{\bb{L}^\infty}^2 \right) h^4
		+
		Ck^2 
		+
		\frac{\alpha}{20} \norm{\bff{\xi}^n}{\bb{L}^2}^2.
	\end{align*}
	Next, for the term $I_{11}$, by Young's inequality, the definition of $g$, and \eqref{equ:Ritz ineq},
	\begin{align*}
		\abs{{I_{11}}}
		&\leq
		C\abs{e^n}^2 + C\norm{g(\bff{u}_h^{n-1})}{\bb{L}^2}^2 \norm{\dtt \bff{\rho}^n}{\bb{L}^2}^2
		\leq
		C\abs{e^n}^2 + Ch^4.
	\end{align*}
	For the next term, by \eqref{equ:g uhn F}, \eqref{equ:dt vn Lp}, \eqref{equ:Ritz ineq}, and Young's inequality we obtain
	\begin{align*}
		\abs{I_{12}}
		&\leq
		C\norm{\dtt \bff{u}^n}{\bb{L}^2}^2 \abs{e^n}^2
		+
		C\norm{\bff{\theta}^{n-1}+\bff{\rho}^{n-1}}{\bb{L}^6}^2
		\leq
		C\abs{e^n}^2 + C\norm{\bff{\theta}^{n-1}}{\bb{H}^1}^2 + Ch^4.
	\end{align*}
	Finally, for the last term, by Young's inequality, \eqref{equ:dt vn min par v}, and the assumed regularity for $\bff{u}$ in \eqref{equ:reg u euler}, we infer that
	\begin{align*}
		\abs{I_{13}}
		\leq
		C\abs{e^n}^2 + Ck^2.
	\end{align*}
	Altogether, we collect these estimates and continue from \eqref{equ:ineq theta e}. Summing over $j\in \{1,2,\ldots, n\}$ and rearranging the terms, we obtain for sufficiently small $k>0$,
	\begin{align*}
		&\norm{\bff{\theta}^n}{\bb{H}^1}^2 + \abs{e^n}^2
		+
		\alpha k \sum_{j=1}^n \norm{\bff{\xi}^j}{\bb{L}^2}^2
		\\
		&\leq
		Ck(h^4+k^2) \sum_{j=1}^{n-1} \left(1+\norm{\bff{u}_h^j}{\bb{L}^\infty}^2 \right)
		+
		Ck \sum_{j=1}^{n-1} \left(1+\norm{\bff{u}_h^j}{\bb{L}^\infty}^2 \right) \norm{\bff{\theta}^j}{\bb{H}^1}^2
		+
		Ck \sum_{j=1}^{n-1} \abs{e^j}^2.
	\end{align*}
	By the discrete Gronwall lemma, noting the stability estimate \eqref{equ:C infty euler}, we deduce
	\begin{align*}
		\norm{\bff{\theta}^n}{\bb{H}^1}^2 + \abs{e^n}^2
		+
		\alpha k \sum_{j=1}^n \norm{\bff{\xi}^j}{\bb{L}^2}^2
		&\leq
		C(h^4+k^2) \exp\left[Ck\sum_{j=1}^{n-1} \left(1+\norm{\bff{u}_h^j}{\bb{L}^\infty}^2 \right) \right]
		\\
		&\leq
		C(h^4+k^2),
	\end{align*}
	proving \eqref{equ:error theta n eul}. 
	
	Finally, estimate \eqref{equ:error theta Linfty eul} follows from \eqref{equ:error theta n eul} and the discrete 2D Sobolev inequality~\eqref{equ:disc sob 2d}. This completes the proof of the proposition.
\end{proof}

The following theorem establishes an optimal rate of convergence for scheme~\eqref{equ:fem euler} in various norms.

\begin{theorem}\label{the:err euler}
	Let $\bff{u}_h^n$ be given by Algorithm~\ref{alg:euler}, and let $\bff{u}$ be the solution of \eqref{equ:llb sav a}. For sufficiently small $h,k>0$, and $n=1,2,\ldots,\lfloor T/k \rfloor$,
	\begin{align}\label{equ:error Hs eul}
		\norm{\bff{u}_h^n- \bff{u}(t_n)}{\bb{H}^s}
		\leq
		C(h^{2-s}+k),\quad s\in \{0,1\}.
	\end{align}
	If $\mathscr{D}\subset \bb{R}^2$, then
	\begin{align}\label{equ:error Linf eul}
		\norm{\bff{u}_h^n- \bff{u}(t_n)}{\bb{L}^\infty}
		\leq
		C(h^2+k) \abs{\ln h}^{\frac12}.
	\end{align}
	Furthermore, the energy functional $\widetilde{\mathcal{E}}[\bff{u}_h^n]$ is a good approximation of $\mathcal{E}[\bff{u}(t_n)]$ in the sense that
	\begin{align}\label{equ:good ener eul}
		\abs{\widetilde{\mathcal{E}}[\bff{u}_h^n]- \mathcal{E}[\bff{u}(t_n)]} \leq C(h+k).
	\end{align}
	The constant $C$ in \eqref{equ:error Hs eul}, \eqref{equ:error Linf eul}, and \eqref{equ:good ener eul} {depends on $T$ and the exact solution $\bff{u}$}, but is independent of $n$, $h$, and $k$.
\end{theorem}

\begin{proof}
	The estimate \eqref{equ:error Hs eul} follows from \eqref{equ:split un}, \eqref{equ:error theta n eul}, \eqref{equ:Ritz ineq}, and the triangle inequality. Correspondingly, the estimate \eqref{equ:error Linf eul} follows from \eqref{equ:split un}, \eqref{equ:error theta Linfty eul}, \eqref{equ:Ritz infty ineq}, and the triangle inequality.
	
	Finally, we have
	\begin{align*}
		&\abs{\widetilde{\mathcal{E}}[\bff{u}_h^n]- \mathcal{E}[\bff{u}(t_n)]}
		\\
		&\leq
		\frac12 \int_{\mathscr{D}} \abs{\abs{\nabla \bff{u}_h^n}^2 - \abs{\nabla \bff{u}(t_n)}^2} \dx
		+
		\frac{\kappa}{2} \int_{\mathscr{D}} \abs{\abs{\bff{u}_h^n}^2 - \abs{\bff{u}(t_n)}^2} \dx
		+
		\abs{\abs{r_h^n}^2 - \abs{r(t_n)}^2} 
		\\
		&\leq
		C\norm{\bff{u}_h^n- \bff{u}(t_n)}{\bb{H}^1}
		+
		C \abs{r_h^k-r(t_k)},
	\end{align*}
	where in the last step we used the elementary inequality $\abs{\abs{\bff{a}}^2-\abs{\bff{b}}^2} \leq \left(\abs{\bff{a}}+ \abs{\bff{b}}\right) \abs{\bff{a}-\bff{b}}$ together with H\"older's inequality. This implies \eqref{equ:good ener eul}, in view of \eqref{equ:error Hs eul}, as required.
\end{proof}

\section{SAV-FEM with linearised BDF2 time discretisation}\label{sec:sav fem cn}

In this section, we construct a linear fully-discrete SAV FEM based on BDF2 time discretisation. Let us define some notations. Firstly, for any function $\bff{v}$, we write $\bff{v}^n:= \bff{v}(t_n)$. Secondly, for a sequence of discrete functions $\{\bff{v}_h^n\}$, we write
\begin{align}\label{equ:vh bar}
	\overline{\bff{v}}_h^{n-1}:= 2\bff{v}_h^{n-1} - \bff{v}_h^{n-2}.
\end{align}
For $n\geq 2$, we also define
\begin{align}\label{equ:Dt vn}
	\Dtt \bff{v}^n:= \dtt \bff{v}^n + \frac12 \left(\dtt \bff{v}^n-\dtt \bff{v}^{n-1}\right) 
	=
	\frac{3\bff{v}^n-4\bff{v}^{n-1}+ \bff{v}^{n-2}}{2k}.
\end{align}

The scheme based on the linearised BDF2 time-stepping is described as follows. We start with $\bff{u}_h^0\in \bb{V}_h$ and $r_h^0= \mathcal{F}(\bff{u}_h^0)\in \bb{R}$, where the functional $\mathcal{F}$ is defined in \eqref{equ:F u}. Given $(\bff{u}_h^{n-1},r_h^{n-1})$ and $(\bff{u}_h^{n-2},r_h^{n-2})$, we seek $(\bff{u}_h^n, r_h^n)\in \bb{V}_h\times \bb{R}$ satisfying
\begin{equation}\label{equ:fem cn}
	\begin{aligned}
		\inpro{\Dtt \bff{u}_h^n}{\phi} 
		&=
		-\gamma \inpro{\overline{\bff{u}}_h^{n-1} \times \bff{H}_h^n}{\phi}
		+\alpha \inpro{\bff{H}_h^n}{\phi}, \quad \forall \bff{\phi}\in \bb{V}_h,
		\\
		\Dtt r_h^n
		&=
		\frac{1}{2\sqrt{\mathcal{F}[\overline{\bff{u}}_h^{n-1}]}} \inpro{g(\overline{\bff{u}}_h^{n-1})}{\Dtt \bff{u}_h^n},
	\end{aligned} 
\end{equation}
where
\begin{align}\label{equ:fem cn Hhn}
	\bff{H}_h^n
	&=
	\Delta_h \bff{u}_h^n - \kappa \bff{u}_h^n - \Pi_h \left[\frac{r_h^n}{\sqrt{\mathcal{F}[\overline{\bff{u}}_h^{n-1}]}}\, g(\overline{\bff{u}}_h^{n-1}) \right].
\end{align}
Here, $\Delta_h$ and $\Pi_h$ are, respectively, the discrete Laplacian and the $\bb{L}^2$-projection operators. The functions $\mathcal{F}$ and $g$ are defined in \eqref{equ:F u} and \eqref{equ:g u}, respectively. 

Note that since \eqref{equ:fem cn} is a two-step scheme, with $\bff{u}_h^0$ and $r_h^0$ given, we also need to define $\bff{u}_h^1$ and $r_h^1$, for instance by performing one step of the implicit BDF2 method or $1/k$ steps of the linearised Euler method \eqref{equ:fem euler} with time-step size $k^2$. For simplicity of presentation, we assume that $\bff{u}_h^1$ and $r_h^1$ have been properly initialised such that
\begin{align}\label{equ:cn init}
	\norm{\bff{u}_h^1- \bff{u}(t_1)}{\bb{H}^s}
	+
	\abs{r_h^1-r(t_1)}
	\leq
	C(h^{2-s}+k^2), \quad s\in \{0,1\},
\end{align}
where $C$ is independent of $h$ and $k$.

The above scheme can be effectively implemented in the following algorithm:

\begin{algorithm}[Linearised BDF2--SAV FEM]\label{alg:bdf}
	Let $h>0$ and $k>0$ be given.
	\\
	\textbf{Input}: Given $\bff{u}_h^0\in \bb{V}_h$ and $r_h^0=\mathcal{F}[\bff{u}_h^0]\in \bb{R}$.
	\begin{enumerate}
		\item Compute $(\bff{u}_h^1, r_h^1)$, say by employing one step of the implicit BDF2 method or $1/k$ steps of the linearised Euler method (Algorithm~\ref{alg:euler}) with time-step size $k^2$.
	\end{enumerate}
	\textbf{For} $n=2$ to $N$, where $N=\lfloor T/k \rfloor$, \textbf{do}:
	\begin{enumerate}
		\item Substitute equation \eqref{equ:fem cn Hhn} into the first equation in \eqref{equ:fem cn}, and compute $\bff{u}_h^n\in \bb{V}_h$;
		\item Compute $r_h^n$ using the second equation in \eqref{equ:fem cn}.
	\end{enumerate}
	\textbf{Output}: a sequence of discrete functions $\{\bff{u}_h^n\}_{1\leq n\leq N}$ and real numbers $\{r_h^n\}_{1\leq n\leq N}$.
\end{algorithm}

Well-posedness of Algorithm~\ref{alg:bdf} follows by the same argument as in Proposition~\ref{pro:well euler}.
The stability and convergence properties for the scheme will be derived next. 
To this end, we first define a modified energy $\widehat{\mathcal{E}}[\bff{u}_h^n]$ of the finite element solution $\bff{u}_h^n$ given by
\begin{align}\label{equ:modif ener cn}
	\widehat{\mathcal{E}}[\bff{u}_h^n]
	&:=
	\widetilde{\mathcal{E}}[\bff{u}_h^n]
	+
	\frac12 \norm{2\nabla \bff{u}_h^n- \nabla \bff{u}_h^{n-1}}{\bb{L}^2}^2
	+
	\frac{\kappa}{2} \norm{2\bff{u}_h^n- \bff{u}_h^{n-1}}{\bb{L}^2}^2
	+
	\abs{2r_h^n- r_h^{n-1}}^2,
\end{align}
with $\widetilde{\mathcal{E}}[\bff{u}_h^n]$ defined in \eqref{equ:ener modified}. We also assume the following regularity for the exact solution $\bff{u}$:
\begin{align}\label{equ:reg u cn}
	\bff{u}\in L^\infty_T(\bb{W}^{2,\infty}) \cap W^{1,\infty}_T(\bb{H}^1) \cap W^{2,\infty}_T(\bb{H}^1) \cap W^{3,\infty}_T(\bb{L}^2).
\end{align} 

In the analysis, the following identity will be useful: for any $\bff{a},\bff{b},\bff{c}\in \bb{R}^3$,
\begin{align}\label{equ:a abc}
	2\bff{a} \cdot (3\bff{a}-4\bff{b}+\bff{c})
	=
	\abs{\bff{a}}^2 - \abs{\bff{b}}^2
	+
	\abs{2\bff{a}-\bff{b}}^2
	-
	\abs{2\bff{b}-\bff{c}}^2
	+
	\abs{\bff{a}-2\bff{b}+\bff{c}}^2.
\end{align}
The following unconditional energy stability result can now be established.

\begin{proposition}
	Let $\bff{u}_h^n$ be given by Algorithm~\ref{alg:bdf}. For $n=1,2,\ldots,\lfloor T/k \rfloor$,
	\begin{align}\label{equ:disc ener cn}
		\widehat{\mathcal{E}}[\bff{u}_h^n] \leq \widehat{\mathcal{E}}[\bff{u}_h^{n-1}].
	\end{align}
	Here, $\widehat{\mathcal{E}}[\bff{u}_h^n]$ is a modified energy of the finite element solution $\bff{u}_h^n$ defined in \eqref{equ:modif ener cn}.
\end{proposition}

\begin{proof}
	Setting $\bff{\phi}= \bff{H}_h^n$ in \eqref{equ:fem cn}, taking the inner product of \eqref{equ:fem cn Hhn} with $\Dtt \bff{u}_h^n$, and multiplying the second equation in \eqref{equ:fem cn} with $2r_h^n$, then combining the resulting equations, noting \eqref{equ:a abc}, we obtain
	\begin{align}\label{equ:hat E min 1}
		&\widehat{\mathcal{E}}[\bff{u}_h^n]- \widehat{\mathcal{E}}[\bff{u}_h^{n-1}]
		+
		\alpha k\norm{\bff{H}_h^n}{\bb{L}^2}^2
		+
		\frac12
		\norm{\nabla\bff{u}_h^n-2 \nabla\bff{u}_h^{n-1}+\nabla\bff{u}_h^{n-2}}{\bb{L}^2}^2
		\nonumber\\
		&\quad
		+
		\frac{\kappa}{2} \norm{\bff{u}_h^n-2 \bff{u}_h^{n-1}+\bff{u}_h^{n-2}}{\bb{L}^2}^2
		+
		\abs{r_h^n -2r_h^{n-1}+ r_h^{n-2}}^2
		= 0,
	\end{align}
	which implies \eqref{equ:disc ener cn}.
\end{proof}

\begin{remark}
	We note that $\widehat{\mathcal{E}}[\bff{u}_h^n]$ defined in \eqref{equ:modif ener cn} can be written as 
	\[
	\widehat{\mathcal{E}}[\bff{u}_h^n]
	:=
	\widetilde{\mathcal{E}}[\bff{u}_h^n]
	+
	2\norm{\nabla \bff{u}_h^{n-\frac12}- \nabla \bff{u}_h^{n-1}}{\bb{L}^2}^2
	+
	2\kappa \norm{\bff{u}_h^{n-\frac12}- \bff{u}_h^{n-1}}{\bb{L}^2}^2
	+
	\left(r_h^{n-\frac12}- r_h^{n-1} \right),
	\]
	where $\bff{v}_h^{n-\frac12}:= \frac12 \left(\bff{v}_h^n+\bff{v}_h^{n-1}\right)$ for any $\bff{v}_h^n\in \bb{V}_h$. Given sufficient regularity of the exact solution, this energy functional is an approximation to the original energy $\mathcal{E}$ as $h,k\to 0^+$.
\end{remark}

Next, we show the following stability estimate.

\begin{lemma}
Let $\bff{u}_h^n$ be given by Algorithm~\ref{alg:bdf}. For $n=1,2,\ldots,\lfloor T/k \rfloor$, we have
\begin{align}\label{equ:uhn H1 stab cn}
	\norm{\bff{u}_h^n}{\bb{H}^1}^2 + \abs{r_h^n}^2 + k \sum_{j=1}^n \norm{\bff{H}_h^j}{\bb{L}^2}^2 \leq C_\mathrm{B}.
\end{align}
Consequently, we have
\begin{align}\label{equ:C infty cn}
	k \sum_{j=1}^\infty \norm{\bff{u}_h^j}{\bb{L}^\infty}^2 \leq C_\infty,
\end{align}
where $C_\mathrm{B}$ and $C_\infty$ are constants {depending on $T$ and the exact solution $\bff{u}$}, but are independent of $n$, $h$, and $k$.
\end{lemma}

\begin{proof}
Summing \eqref{equ:hat E min 1} over $j\in \{1,2,\ldots, n\}$, we obtain \eqref{equ:uhn H1 stab cn}. Finally, the proof of \eqref{equ:C infty cn} follows similarly to that of \eqref{equ:C infty euler}, but now makes use of \eqref{equ:fem cn Hhn} and \eqref{equ:uhn H1 stab cn} instead.
\end{proof}

To facilitate the proof of the error analysis, we split the approximation error as before, namely:
\begin{align}
	\label{equ:split un cn}
	\bff{u}_h^n-\bff{u}^n &= (\bff{u}_h^n -R_h \bff{u}^n) + (R_h \bff{u}^n- \bff{u}^n) =: \bff{\theta}^n + \bff{\rho}^n,
	\\
	\label{equ:split Hn cn}
	\bff{H}_h^n-\bff{H}^n &= (\bff{H}_h^n -R_h \bff{H}^n) + (R_h \bff{H}^n- \bff{H}^n) =: \bff{\xi}^n + \bff{\eta}^n,
	\\
	\label{equ:split rn cn}
	r_h^n-r^n &= e^n,
\end{align}
where $R_h$ is the Ritz projection defined by \eqref{equ:Ritz}.
We will also make use of the quantities $\overline{\bff{\theta}}^n, \overline{\bff{\rho}}^n, \overline{\bff{\xi}}^n$, and $\overline{\bff{\eta}}^n$ which are defined similarly, noting \eqref{equ:vh bar}. 

Several elementary estimates are collected next. Assume the regularity of $\bff{u}$ stated in \eqref{equ:reg u cn}. For $p\in (1,\infty)$, there exists a positive constant $C_p$ such that
\begin{align}
	%	\label{equ:dtt vn Lp}
	%	\norm{\dtt^2 \bff{v}^n}{\bb{L}^p}
	%	&=
	%	\norm{\frac{1}{k} \left(\dtt \bff{v}^n- \dtt \bff{v}^{n-1}\right)}{\bb{L}^p}
	%	\leq
	%	C_p \norm{\bff{v}}{W^{2,\infty}_T(\bb{L}^p)},
	%	\\
	\label{equ:Dt bdf vn Lp}
	\norm{\mathcal{D}_t \bff{v}^n}{\bb{L}^p}
	&=
	\norm{\dtt \bff{v}^n+ \frac12 \left(\dtt \bff{v}^n- \dtt \bff{v}^{n-1}\right)}{\bb{L}^p}
	\leq
	C_p,
	\\
	\label{equ:dtv part t v cn}
	\norm{\Dtt \bff{v}^n- \partial_t \bff{v}^n}{\bb{L}^p}
	&\leq
	Ck \int_{t_{n-2}}^{t_n} \norm{\partial_t^3 \bff{v}(t)}{\bb{L}^p} \dt \leq Ck^2 \norm{\bff{v}}{W^{3,\infty}_T(\bb{L}^p)}.
\end{align}
Consequently, for $\bff{\rho}^n$ and $\bff{\eta}^n$ defined in \eqref{equ:split un cn} and \eqref{equ:split Hn cn},
\begin{align}
	\label{equ:Dt bdf rhon Lp}
	\norm{\mathcal{D}_t \bff{\rho}^n}{\bb{L}^p}
	&=
	\norm{\dtt \bff{\rho}^n+ \frac12 \left(\dtt \bff{\rho}^n- \dtt \bff{\rho}^{n-1}\right)}{\bb{L}^p}
	\leq
	C_p h^2,
	\quad \text{and } \;
	\norm{\mathcal{D}_t \bff{\eta}^n}{\bb{L}^p}
	\leq
	C_p h^2.
\end{align}
Further preliminary estimates will be needed in the proof of the main theorem, which we now derive.

\begin{lemma}\label{lem:nonlinear bdf}
	Let $\bff{u}_h^n$ be given by Algorithm~\ref{alg:bdf}. Let $\bff{\theta}^n$ and $\bff{\rho}^n$ be given by \eqref{equ:split un}. Furthermore, let the functions $\mathcal{F}$ and $g$ be given by \eqref{equ:F u} and \eqref{equ:g u}, respectively. Then
	\begin{align}
		\label{equ:g uhn F cn}
		\norm{\frac{g(\overline{\bff{u}}_h^{n-1})}{\sqrt{\mathcal{F}[\overline{\bff{u}}_h^{n-1}]}} - \frac{g(\overline{\bff{u}}^{n-1})}{\sqrt{\mathcal{F}[\overline{\bff{u}}^{n-1}]}}}{\bb{L}^2}
		&\leq
		C \norm{{\bff{\theta}}^{n-1}}{\bb{L}^6}
		+
		C \norm{{\bff{\theta}}^{n-2}}{\bb{L}^6}
		+
		C \norm{\overline{\bff{\rho}}^{n-1}}{\bb{L}^6},
		\\[1ex]
		\label{equ:g un F cn}
		\norm{\frac{g(\bff{u}^n)}{\sqrt{\mathcal{F}[\bff{u}^n]}} - \frac{g(\overline{\bff{u}}^{n-1})}{\sqrt{\mathcal{F}[\overline{\bff{u}}^{n-1}]}}}{\bb{L}^2}
		&\leq
		Ck^2,
	\end{align}
	where $C$ is a constant {depending on $T$ and the exact solution $\bff{u}$}, but is independent of $n$, $h$, and $k$.
\end{lemma}

\begin{proof}
	We first prove \eqref{equ:g uhn F cn}. By the same argument as in \eqref{equ:est gf}, we have
	\begin{align*}
		&\norm{\frac{g(\overline{\bff{u}}_h^{n-1})}{\sqrt{\mathcal{F}[\overline{\bff{u}}_h^{n-1}]}} - \frac{g(\overline{\bff{u}}^{n-1})}{\sqrt{\mathcal{F}[\overline{\bff{u}}^{n-1}]}}}{\bb{L}^2}
		\leq
		C\norm{g(\overline{\bff{u}}_h^{n-1})-g(\overline{\bff{u}}^{n-1})}{\bb{L}^2} + C\abs{\F{\overline{\bff{u}}^{n-1}}- \F{\overline{\bff{u}}_h^{n-1}}}.
	\end{align*}
	The first term on the right-hand side can be estimated as in \eqref{equ:guhn int} to obtain
	\begin{align*}
		\norm{g(\overline{\bff{u}}_h^{n-1})-g(\overline{\bff{u}}^{n-1})}{\bb{L}^2}
		&\leq
		C\left(\norm{\overline{\bff{u}}_h^{n-1}}{\bb{L}^6}^2 + \norm{\overline{\bff{u}}^{n-1}}{\bb{L}^6}^2 \right) \norm{\overline{\bff{\theta}}^{n-1}+ \overline{\bff{\rho}}^{n-1}}{\bb{L}^6}
		\\
		&\leq
		C \norm{\bff{\theta}^{n-1}}{\bb{L}^6}
		+
		C\norm{\bff{\theta}^{n-2}}{\bb{L}^6}
		+
		C \norm{\overline{\bff{\rho}}^{n-1}}{\bb{L}^6}.
	\end{align*}
	By the same argument as in \eqref{equ:guhn int}, we also obtain
	\begin{align*}
		\abs{\F{\overline{\bff{u}}^{n-1}}- \F{\overline{\bff{u}}_h^{n-1}}}
		&\leq
		C\left(\norm{\overline{\bff{u}}_h^{n-1}}{\bb{L}^4}^2 + \norm{\overline{\bff{u}}^{n-1}}{\bb{L}^4}^2 \right) \left(\norm{\overline{\bff{u}}^{n-1}}{\bb{L}^4}+\norm{\overline{\bff{u}}_h^{n-1}}{\bb{L}^4}\right) \norm{\overline{\bff{\theta}}^{n-1}+ \overline{\bff{\rho}}^{n-1}}{\bb{L}^4}
		\\
		&\leq
		C \norm{\bff{\theta}^{n-1}}{\bb{L}^4} + C \norm{\bff{\theta}^{n-2}}{\bb{L}^4} + C \norm{\bff{\overline{\rho}}^{n-1}}{\bb{L}^4}.
	\end{align*}
	Altogether, we infer \eqref{equ:g uhn F cn}.
	
	Next, we show \eqref{equ:g un F cn}. By the same token, we obtain
	\begin{align*}
		\norm{\frac{g(\bff{u}^n)}{\sqrt{\mathcal{F}[\bff{u}^n]}} - \frac{g(\overline{\bff{u}}^{n-1})}{\sqrt{\mathcal{F}[\overline{\bff{u}}^{n-1}]}}}{\bb{L}^2}
		&\leq
		C \norm{\bff{u}^n -\overline{\bff{u}}^{n-1}}{\bb{L}^6}
		\leq
		Ck^2 \norm{\bff{u}}{W^{2,\infty}_T(\bb{H}^1)},
	\end{align*}
	where in the last step we used Taylor's theorem. Noting the regularity of $\bff{u}$, we deduce \eqref{equ:g un F cn}, thus completing the proof of the lemma.
\end{proof}

\begin{lemma}
	Let $\bff{u}_h^n$ be defined by scheme \eqref{equ:fem cn}. Let $\bff{\theta}^n$ and $\bff{\rho}^n$ be given by \eqref{equ:split un cn}, and let $\bff{\xi}^n$ and $\bff{\eta}^n$ be given by \eqref{equ:split Hn cn}. Then
	\begin{align}\label{equ:dt theta L2 cn}
		\norm{\mathcal{D}_t \bff{\theta}^n}{\bb{L}^2}
		&\leq
		C\left(1+\norm{\overline{\bff{u}}_h^{n-1}}{\bb{L}^\infty}\right) \norm{\bff{\xi}^n}{\bb{L}^2}
		+
		C\norm{\overline{\bff{\theta}}^{n-1}}{\bb{L}^2}
		+
		C\left(1+\norm{\overline{\bff{u}}_h^{n-1}}{\bb{L}^\infty}\right) h^2 + Ck^2,
	\end{align}
where $C$ is a constant {depending on $T$ and the exact solution $\bff{u}$}, but is independent of $n$, $h$, and $k$.
%		\\
%		\label{equ:dt theta cn L32}
%		\norm{\Dtt \bff{\theta}^n}{\bb{L}^{3/2}}
%		&\leq
%		Ck^2 + Ch^2 + C \norm{\bff{\xi}^n}{\bb{L}^2}
%		+
%		C\norm{\overline{\bff{\theta}}^{n-1}}{\bb{L}^2}.
		
%		Consequently, we also have
%		\begin{align}\label{equ:dt uh L2 cn}
%		\norm{\Dtt \bff{u}_h^n}{\bb{L}^2}
%		&\leq
%		C\left(1+\norm{\bff{u}_h^{n-1}}{\bb{L}^\infty}\right) \norm{\bff{\xi}^n}{\bb{L}^2}
%		+
%		C\norm{\bff{\theta}^{n-1}}{\bb{L}^2}
%		+
%		C\left(1+\norm{\bff{u}_h^{n-1}}{\bb{L}^\infty}\right) h^2 
%		\nonumber\\
%		&\quad
%		+ Ck
%		+ \norm{u}{W^{2,\infty}_T(\bb{L}^2)},
%		\end{align}
\end{lemma}

\begin{proof}
	Subtracting the weak formulation of \eqref{equ:llb sav eq1} from the corresponding equation in scheme \eqref{equ:fem cn}, we obtain
	\begin{align}
		\label{equ:uhn min un cn}
		&\inpro{\Dtt \bff{\theta}^n + \Dtt \bff{\rho}^n + \Dtt \bff{u}^n - \partial_t \bff{u}^n}{\bff{\phi}}
		\nonumber \\
		&=
		-\gamma \inpro{\overline{\bff{u}}_h^{n-1} \times (\bff{\xi}^n +\bff{\eta}^n)}{\bff{\phi}}
		-
		\gamma \inpro{(\overline{\bff{\theta}}^{n-1}+ \overline{\bff{\rho}}^{n-1})\times \bff{H}^n}{\bff{\phi}}
		\nonumber\\
		&\quad
		-
		\gamma \inpro{(\bff{u}^n-\overline{\bff{u}}^{n-1})\times \bff{H}^n}{\bff{\phi}}
		+
		\alpha \inpro{\bff{\xi}^n +\bff{\eta}^n}{\bff{\phi}}, \quad \forall \bff{\phi}\in \bb{V}_h.
	\end{align} 
	Therefore, we have
	\begin{align*}
		\Dtt \bff{\theta}^n
		&=
		\Pi_h \big[\partial_t \bff{u}^n - \Dtt \bff{u}^n\big] 
		-
		\Pi_h \dtt \bff{\rho}^n
		-
		\gamma \Pi_h \big[\overline{\bff{u}}_h^{n-1} \times (\bff{\xi}^n +\bff{\eta}^n)\big]
		\\
		&\quad
		-
		\gamma \Pi_h\big[(\overline{\bff{\theta}}^{n-1}+ \overline{\bff{\rho}}^{n-1})\times \bff{H}^n\big] 
		-
		\gamma \Pi_h \big[(\bff{u}^n-\overline{\bff{u}}^{n-1})\times \bff{H}^n\big] 
		+
		\alpha \bff{\xi}^n
		+
		\alpha \Pi_h \bff{\eta}^n,
	\end{align*}
	where $\Pi_h$ is the orthogonal projection defined by \eqref{equ:orth proj}.
	Noting the $\bb{L}^2$ stability of $\Pi_h$, we have by H\"older's inequality,
	\begin{align*}
		\norm{\Dtt \bff{\theta}^n}{\bb{L}^2}
		&\leq
		C\norm{\partial_t \bff{u}^n - \Dtt \bff{u}^n}{\bb{L}^2}
		+
		C\norm{\Dtt \bff{\rho}^n}{\bb{L}^2}
		+
		C\norm{\overline{\bff{u}}_h^{n-1}}{\bb{L}^\infty} \norm{\bff{\xi}^n +\bff{\eta}^n}{\bb{L}^2}
		\\
		&\quad
		+
		C\norm{\overline{\bff{\theta}}^{n-1}+ \overline{\bff{\rho}}^{n-1}}{\bb{L}^2} \norm{\bff{H}^n}{\bb{L}^\infty}
		+
		C\norm{\bff{u}^n-\overline{\bff{u}}^{n-1}}{\bb{L}^2} \norm{\bff{H}^n}{\bb{L}^\infty}
		+
		C\norm{\bff{\xi}^n}{\bb{L}^2}
		+
		C\norm{\bff{\eta}^n}{\bb{L}^2}
		\\
		&\leq
		Ck^2 + Ch^2 + C \norm{\overline{\bff{u}}_h^{n-1}}{\bb{L}^\infty} \norm{\bff{\xi}^n}{\bb{L}^2}
		+
		Ch^2 \norm{\overline{\bff{u}}_h^{n-1}}{\bb{L}^\infty}
		+
		C\norm{\overline{\bff{\theta}}^{n-1}}{\bb{L}^2}
		+
		C\norm{\bff{\xi}^n}{\bb{L}^2},
	\end{align*}
	where in the last step we used \eqref{equ:dtv part t v cn}, \eqref{equ:Ritz ineq}, and the regularity of the solution assumed in \eqref{equ:reg u cn}. This implies \eqref{equ:dt theta L2 cn}, completing the proof of the lemma.
\end{proof}
	
%	Similarly, we also have
%	\begin{align*}
%		\norm{\Dtt \bff{\theta}^n}{\bb{L}^{3/2}}
%		&\leq
%		C\norm{\partial_t \bff{u}^n - \Dtt \bff{u}^n}{\bb{L}^{3/2}}
%		+
%		C\norm{\Dtt \bff{\rho}^n}{\bb{L}^{3/2}}
%		+
%		C\norm{\overline{\bff{u}}_h^{n-1}}{\bb{L}^6} \norm{\bff{\xi}^n +\bff{\eta}^n}{\bb{L}^2}
%		\\
%		&\quad
%		+
%		C\norm{\overline{\bff{\theta}}^{n-1}+ \overline{\bff{\rho}}^{n-1}}{\bb{L}^2} \norm{\bff{H}^n}{\bb{L}^6}
%		+
%		C\norm{\bff{u}^n-\overline{\bff{u}}^{n-1}}{\bb{L}^2} \norm{\bff{H}^n}{\bb{L}^6}
%		+
%		C\norm{\bff{\xi}^n}{\bb{L}^2}
%		+
%		C\norm{\bff{\eta}^n}{\bb{L}^2}
%		\\
%		&\leq
%		Ck^2 + Ch^2 + C \norm{\bff{\xi}^n}{\bb{L}^2}
%		+
%		C\norm{\overline{\bff{\theta}}^{n-1}}{\bb{L}^2},
%	\end{align*}
%	proving \eqref{equ:dt theta cn L32}.
	
%	Finally, inequality \eqref{equ:dt uh L2 cn} follows by writing $\Dtt \bff{u}_h^n= \Dtt \bff{\theta}^n+ \Dtt \bff{\rho}^n + \Dtt \bff{u}^n$, then applying \eqref{equ:Dt bdf rhon Lp}, \eqref{equ:Dt bdf vn Lp}, \eqref{equ:dt theta L2 cn}, and the triangle inequality. This completes the proof of the lemma.

The following proposition establishes an auxiliary superconvergence estimate.

\begin{proposition}
	Let $\bff{\theta}^n, \bff{\xi}^n$, and $e^n$ be defined by \eqref{equ:split un cn}, \eqref{equ:split Hn cn}, and \eqref{equ:split rn cn}, respectively. For sufficiently small $h,k>0$, we have 
	\begin{align}\label{equ:error theta n cn}
		\norm{\bff{\theta}^n}{\bb{H}^1}^2 + \abs{e^n}^2 
		\leq C(h^4+k^4).
	\end{align}
	Moreover, if $\mathscr{D}\subset \bb{R}^2$, then
	\begin{align}\label{equ:error theta Linfty cn}
		\norm{\bff{\theta}^n}{\bb{L}^\infty}^2
		\leq
		C(h^4+k^4)\abs{\ln h}.
	\end{align}
	Here, $C$ is a constant {depending on $T$ and the exact solution $\bff{u}$}, but is independent of $n$, $h$, and $k$.
\end{proposition}

\begin{proof}
	Firstly, recall that by subtracting the weak formulation of \eqref{equ:llb sav eq1} at time $t=t_n$ from the corresponding equation in scheme \eqref{equ:fem cn}, we have \eqref{equ:uhn min un cn}.
	In a similar manner, noting \eqref{equ:Ritz zero}, we have
	\begin{align}
		\label{equ:Hhn min Hn cn}
		\inpro{\bff{\xi}^n +\bff{\eta}^n}{\bff{\chi}}
		&=
		- \inpro{\nabla \bff{\theta}^n}{\nabla \bff{\chi}}
		-
		\kappa \inpro{\bff{\theta}^n +\bff{\rho}^n}{\bff{\chi}}
		\nonumber\\
		&\quad
		-
		\frac{e^n}{\sqrt{\mathcal{F}[\overline{\bff{u}}^{n-1}]}} \inpro{g(\overline{\bff{u}}^{n-1})}{\bff{\chi}}
		-
		r_h^n \inpro{\frac{g(\overline{\bff{u}}_h^{n-1})}{\sqrt{\mathcal{F}[\overline{\bff{u}}_h^{n-1}]}}- \frac{g(\overline{\bff{u}}^{n-1})}{\sqrt{\mathcal{F}[\overline{\bff{u}}^{n-1}]}}}{\bff{\chi}}
		\nonumber\\
		&\quad
		+
		r^n \inpro{\frac{g(\bff{u}^n)}{\sqrt{\mathcal{F}[\bff{u}^n]}} - \frac{g(\overline{\bff{u}}^{n-1})}{\sqrt{\mathcal{F}[\overline{\bff{u}}^{n-1}]}}}{\bff{\chi}}, \quad \forall \bff{\chi}\in \bb{V}_h.
	\end{align}
	Next, subtracting \eqref{equ:llb sav eq3} from the second equation in \eqref{equ:fem cn}, we obtain
	\begin{align}\label{equ:rhn min rn cn}
		&\Dtt e^n + \Dtt r^n - \partial_t r^n 
		\nonumber\\
		&=
		\frac{1}{2\sqrt{\mathcal{F}[\overline{\bff{u}}_h^{n-1}]}} \inpro{g(\overline{\bff{u}}_h^{n-1})}{\Dtt \bff{\theta}^n+ \Dtt \bff{\rho}^n}
		+
		\inpro{\frac{g(\overline{\bff{u}}_h^{n-1})}{2\sqrt{\mathcal{F}[\overline{\bff{u}}_h^{n-1}]}} - \frac{g(\overline{\bff{u}}^{n-1})}{2\sqrt{\mathcal{F}[\overline{\bff{u}}^{n-1}]}}}{\Dtt \bff{u}^n}
		\nonumber\\
		&\quad
		+
		\inpro{\frac{g(\overline{\bff{u}}^{n-1})}{2\sqrt{\mathcal{F}[\overline{\bff{u}}^{n-1}]}} - \frac{g(\bff{u}^n)}{2\sqrt{\mathcal{F}[\bff{u}^n]}}}{\Dtt \bff{u}^n}
		+
		\inpro{\frac{g(\bff{u}^n)}{2\sqrt{\mathcal{F}[\bff{u}^n]}}}{\Dtt \bff{u}^n-\partial_t \bff{u}^n}.
	\end{align}
	Now, set $\bff{\phi}= \bff{\xi}^n$ in \eqref{equ:uhn min un cn} and $\bff{\chi}= \Dtt \bff{\theta}^n$ in \eqref{equ:Hhn min Hn cn}, and multiply \eqref{equ:rhn min rn cn} by $2e^n$. We then add the resulting equations to obtain
	\begin{align}\label{equ:ineq theta e cn}
		&\frac{1}{2k} \left(\norm{\nabla \bff{\theta}^n}{\bb{L}^2}^2 - \norm{\nabla \bff{\theta}^{n-1}}{\bb{L}^2}^2 \right) 
		+
		\frac{\kappa}{2k} \left(\norm{\bff{\theta}^n}{\bb{L}^2}^2 - \norm{\bff{\theta}^{n-1}}{\bb{L}^2}^2 \right)
		+
		\frac{1}{k}  \left(\abs{e^n}^2- \abs{e^{n-1}}^2\right)
		+
		\alpha \norm{\bff{\xi}^n}{\bb{L}^2}^2
		\nonumber\\
		&\quad
		+
		\frac{1}{2k} \left(\norm{2\nabla \bff{\theta}^n- \nabla \bff{\theta}^{n-1}}{\bb{L}^2}^2 - \norm{2\nabla \bff{\theta}^{n-1}- \nabla \bff{\theta}^{n-2}}{\bb{L}^2}^2 \right)
		\nonumber\\
		&\quad
		+
		\frac{\kappa}{2k} \left(\norm{2\bff{\theta}^n-  \bff{\theta}^{n-1}}{\bb{L}^2}^2 - \norm{2 \bff{\theta}^{n-1}- \bff{\theta}^{n-2}}{\bb{L}^2}^2 \right)
		\nonumber\\
		&\quad
		+
		\frac{1}{k} \left(\abs{2e^n-e^{n-1}}^2 - \abs{2e^{n-1}- e^{n-2}}^2 \right) 
		+
		\frac{1}{2k} \norm{\nabla \bff{\theta}^n- 2\nabla \bff{\theta}^{n-1}+ \nabla \bff{\theta}^{n-2}}{\bb{L}^2}^2
		\nonumber\\
		&\quad
		+
		\frac{\kappa}{2k} \norm{\bff{\theta}^n- 2 \bff{\theta}^{n-1}+ \bff{\theta}^{n-2}}{\bb{L}^2}^2
		+
		\frac{1}{k} \abs{e^n-2e^{n-1}+e^{n-2}}^2
		\nonumber\\
		&=
		\inpro{\Dtt \bff{\rho}^n+\Dtt \bff{u}^n-\partial_t \bff{u}^n}{\bff{\xi}^n}
		-
		\inpro{\bff{\eta}^n}{\Dtt \bff{\theta}^n}
		+
		\gamma \inpro{\overline{\bff{u}}_h^{n-1}\times \bff{\eta}^n}{\bff{\xi}^n}
		\nonumber\\
		&\quad
		+
		\gamma \inpro{(\overline{\bff{\theta}}^{n-1}+\overline{\bff{\rho}}^{n-1})\times \bff{H}^n}{\bff{\xi}^n}
		+
		\gamma \inpro{(\bff{u}^n-\overline{\bff{u}}^{n-1})\times \bff{H}^n}{\bff{\xi}^n}
		-
		\alpha \inpro{\bff{\eta}^n}{\bff{\xi}^n}
		-
		\kappa \inpro{\bff{\rho}^n}{\Dtt \bff{\theta}^n}
		\nonumber\\
		&\quad
		-
		r_h^n \inpro{\frac{g(\overline{\bff{u}}_h^{n-1})}{\sqrt{\mathcal{F}[\overline{\bff{u}}_h^{n-1}]}}- \frac{g(\overline{\bff{u}}^{n-1})}{\sqrt{\mathcal{F}[\overline{\bff{u}}^{n-1}]}}}{\Dtt \bff{\theta}^n}
		+
		r^n \inpro{\frac{g(\bff{u}^n)}{\sqrt{\mathcal{F}[\bff{u}^n]}} - \frac{g(\overline{\bff{u}}^{n-1})}{\sqrt{\mathcal{F}[\overline{\bff{u}}^{n-1}]}}}{\Dtt \bff{\theta}^n}
		\nonumber\\
		&\quad
		+
		e^n \inpro{\frac{g(\overline{\bff{u}}_h^{n-1})}{\sqrt{\mathcal{F}[\overline{\bff{u}}_h^{n-1}]}}- \frac{g(\overline{\bff{u}}^{n-1})}{\sqrt{\mathcal{F}[\overline{\bff{u}}^{n-1}]}}}{\Dtt \bff{\theta}^n}
		+
		\frac{e^n}{\sqrt{\mathcal{F}[\overline{\bff{u}}_h^{n-1}]}} \inpro{g(\overline{\bff{u}}_h^{n-1})}{\Dtt \bff{\rho}^n}
		\nonumber\\
		&\quad
		+
		e^n
		\inpro{\frac{g(\overline{\bff{u}}_h^{n-1})}{\sqrt{\mathcal{F}[\overline{\bff{u}}_h^{n-1}]}} - \frac{g(\overline{\bff{u}}^{n-1})}{\sqrt{\mathcal{F}[\overline{\bff{u}}^{n-1}]}}}{\Dtt \bff{u}^n}
		+
		e^n
		\inpro{\frac{g(\overline{\bff{u}}^{n-1})}{\sqrt{\mathcal{F}[\overline{\bff{u}}^{n-1}]}} - \frac{g(\bff{u}^n)}{\sqrt{\mathcal{F}[\bff{u}^n]}}}{\Dtt \bff{u}^n}
		\nonumber\\
		&\quad
		+
		e^n \inpro{\frac{g(\bff{u}^n)}{\sqrt{\mathcal{F}[\bff{u}^n]}}}{\Dtt \bff{u}^n-\partial_t \bff{u}^n}
		\nonumber\\
		&=: J_1+J_2+\ldots+J_{14}.
	\end{align}
	We shall estimate each term on the last line. For the first term, by Young's inequality, \eqref{equ:Ritz ineq}, and \eqref{equ:dtv part t v cn}, we have
	\begin{align*}
		\abs{J_1}
		&\leq
		Ch^4 + Ck^4 + \frac{\alpha}{20} \norm{\bff{\xi}^n}{\bb{L}^2}^2.
	\end{align*}
	Next, noting \eqref{equ:dt theta L2 cn}, we have for the second term,
	\begin{align*}
		\abs{J_2}
		&\leq C\norm{\overline{\bff{\theta}}^{n-1}}{\bb{L}^2}^2
		+
		C\left(1+\norm{\overline{\bff{u}}_h^{n-1}}{\bb{L}^\infty}^2 \right) h^4
		+
		Ck^4
		+
		\frac{\alpha}{20} \norm{\bff{\xi}^n}{\bb{L}^2}^2.
	\end{align*}
	Similarly, for the third term, by \eqref{equ:Ritz ineq} and Young's inequality,
	\begin{align*}
		\abs{J_3}
		&\leq
		Ch^4 \norm{\overline{\bff{u}}_h^{n-1}}{\bb{L}^\infty}^2
		+
		\frac{\alpha}{20} \norm{\bff{\xi}^n}{\bb{L}^2}^2.
	\end{align*}
	For the next three terms, we apply Young's inequality and \eqref{equ:Ritz ineq}, and note the regularity of the exact solution in \eqref{equ:reg u cn} to obtain
	\begin{align*}
		\abs{J_4}
		&\leq
		C\norm{\overline{\bff{\theta}}^{n-1}}{\bb{L}^2}^2
		+
		Ch^4
		+
		\frac{\alpha}{20} \norm{\bff{\xi}^n}{\bb{L}^2}^2,
		\\
		\abs{J_5}
		&\leq
		Ck^4
		+
		\frac{\alpha}{20} \norm{\bff{\xi}^n}{\bb{L}^2}^2,
		\\
		\abs{J_6}
		&\leq
		Ch^4
		+
		\frac{\alpha}{20} \norm{\bff{\xi}^n}{\bb{L}^2}^2.
	\end{align*}
	Next, by Young's inequality and \eqref{equ:dt theta L2 cn},
	\begin{align*}
		\abs{J_7}
		&\leq
		C\norm{\overline{\bff{\theta}}^{n-1}}{\bb{L}^2}^2
		+
		C\left(1+\norm{\overline{\bff{u}}_h^{n-1}}{\bb{L}^\infty}^2 \right) h^4 + Ck^4 + \frac{\alpha}{20} \norm{\bff{\xi}^n}{\bb{L}^2}^2.
	\end{align*}
	For the term $J_8$, by \eqref{equ:g uhn F cn}, \eqref{equ:uhn H1 stab cn}, \eqref{equ:dt theta L2 cn}, as well as H\"older's and Young's inequalities, we have
	\begin{align*}
		\abs{J_8}
		&\leq
		C \abs{r_h^n} \norm{\frac{g(\overline{\bff{u}}_h^{n-1})}{\sqrt{\mathcal{F}[\overline{\bff{u}}_h^{n-1}]}} - \frac{g(\overline{\bff{u}}^{n-1})}{\sqrt{\mathcal{F}[\overline{\bff{u}}^{n-1}]}}}{\bb{L}^2}  \norm{\Dtt \bff{\theta}^n}{\bb{L}^2}
		\\
		&\leq
		C\left( \norm{\bff{\theta}^{n-1}}{\bb{H}^1} + \norm{\bff{\theta}^{n-2}}{\bb{H}^1} + h^2 \right)
		\left[ \left(1+\norm{\overline{\bff{u}}_h^{n-1}}{\bb{L}^\infty}\right) \norm{\bff{\xi}^n}{\bb{L}^2} \right]
		\\
		&\quad
		+ C\left( \norm{\bff{\theta}^{n-1}}{\bb{H}^1} + \norm{\bff{\theta}^{n-2}}{\bb{H}^1} + h^2 \right)
		\left[ \norm{\bff{\theta}^{n-1}}{\bb{L}^2} + \norm{\bff{\theta}^{n-2}}{\bb{L}^2} \right]
		\\
		&\quad
		+ C\left( \norm{\bff{\theta}^{n-1}}{\bb{H}^1} + \norm{\bff{\theta}^{n-2}}{\bb{H}^1} + h^2 \right)
		\left[\left(1+\norm{\overline{\bff{u}}_h^{n-1}}{\bb{L}^\infty}\right) h^2 + Ck^2 \right]
		\\
		&\leq
		C\left(1+\norm{\overline{\bff{u}}_h^{n-1}}{\bb{L}^\infty}^2 \right) \left(\norm{\bff{\theta}^{n-1}}{\bb{H}^1}^2 + \norm{\bff{\theta}^{n-2}}{\bb{H}^1}^2 \right)
		+
		C\left(1+\norm{\overline{\bff{u}}_h^{n-1}}{\bb{L}^\infty}^2 \right) h^4
		+
		Ck^4 
		+
		\frac{\alpha}{20} \norm{\bff{\xi}^n}{\bb{L}^2}^2.
	\end{align*}
	Similarly, applying \eqref{equ:g un F cn} instead, we obtain
	\begin{align*}
		\abs{J_9}
		&\leq
		C\left(1+\norm{\overline{\bff{u}}_h^{n-1}}{\bb{L}^\infty}^2\right) (h^4+k^4)
		+
		C\norm{\overline{\bff{\theta}}^{n-1}}{\bb{L}^2}^2
		+
		\frac{\alpha}{20} \norm{\bff{\xi}^n}{\bb{L}^2}^2.
	\end{align*}
	The term $J_{10}$ can be estimated in the same way as the term $J_8$, yielding the same bound:
	\begin{align*}
		\abs{J_{10}}
		&\leq
		C\left(1+\norm{\overline{\bff{u}}_h^{n-1}}{\bb{L}^\infty}^2 \right) \left(\norm{\bff{\theta}^{n-1}}{\bb{H}^1}^2 + \norm{\bff{\theta}^{n-2}}{\bb{H}^1}^2 \right)
		+
		C\left(1+\norm{\overline{\bff{u}}_h^{n-1}}{\bb{L}^\infty}^2 \right) h^4
		+
		Ck^4 
		+
		\frac{\alpha}{20} \norm{\bff{\xi}^n}{\bb{L}^2}^2.
	\end{align*}
	Next, by Young's inequality and the Sobolev embedding, \eqref{equ:uhn H1 stab cn}, \eqref{equ:Ritz ineq}, and the fact that $\mathcal{F}$ is bounded below, we obtain
	\begin{align*}
		\abs{J_{11}}
		&\leq
		C\norm{\overline{\bff{u}}_h^{n-1}}{\bb{L}^6}^6 \abs{e^n}^2 + Ch^4
		\leq
		C \abs{e^n}^2 + Ch^4.
	\end{align*}
	The terms $J_{12}$ and $J_{13}$ can be bounded by using \eqref{equ:g uhn F cn} and \eqref{equ:g un F cn}, respectively, and applying Young's inequality to yield
	\begin{align*}
		\abs{J_{12}}
		&\leq
		C\norm{\bff{\theta}^{n-1}}{\bb{H}^1}^2
		+
		C\norm{\bff{\theta}^{n-2}}{\bb{H}^1}^2
		+
		C\abs{e^n}^2
		+
		Ch^4,
		\\
		\abs{J_{13}}
		&\leq
		C\abs{e^n}^2
		+
		Ck^4.
	\end{align*}
	Finally, the last term in \eqref{equ:ineq theta e cn} can be bounded using \eqref{equ:dtv part t v cn} and Young's inequality, giving
	\begin{align*}
		\abs{J_{14}}
		&\leq
		C\abs{e^n}^2
		+
		Ck^4.
	\end{align*}
	Altogether, substituting these estimates into \eqref{equ:ineq theta e cn}, rearranging the terms, and summing over $j\in \{2,3,\ldots,n\}$, we obtain for sufficiently small $k>0$,
	\begin{align*}
		&\norm{\bff{\theta}^n}{\bb{H}^1}^2
		+
		\abs{e^n}^2
		+
		\norm{2\bff{\theta}^n-\bff{\theta}^{n-1}}{\bb{H}^1}^2
		+
		\abs{2e^n-e^{n-1}}^2
		\\
		&\leq
		\norm{\bff{\theta}^1}{\bb{H}^1}^2
		+
		\abs{e^1}^2
		+
		\norm{2\bff{\theta}^1-\bff{\theta}^0}{\bb{H}^1}^2
		+
		\abs{2e^1-e^0}^2
		\\
		&\quad
		+
		Ck(h^4+k^4) \sum_{j=2}^n \left(1+\norm{\overline{\bff{u}}_h^{j-1}}{\bb{L}^\infty}^2 \right)
		+
		Ck \sum_{j=2}^n \left(1+\norm{\overline{\bff{u}}_h^{j-1}}{\bb{L}^\infty}^2 \right) \left(\norm{\bff{\theta}^{j-1}}{\bb{H}^1}^2 + \norm{\bff{\theta}^{j-2}}{\bb{H}^1}^2 \right)
		\\
		&\leq
		C(h^4+k^4)
		+
		Ck \sum_{j=1}^{n-1} \left(1+\norm{\bff{u}_h^j}{\bb{L}^\infty}^2 +\norm{\bff{u}_h^{j-1}}{\bb{L}^\infty}^2 \right) \left(\norm{\bff{\theta}^j}{\bb{H}^1}^2 + \norm{\bff{\theta}^{j-1}}{\bb{H}^1}^2 \right),
	\end{align*}
	where in the last step we used \eqref{equ:cn init} and \eqref{equ:C infty cn}. 
	Therefore, by the discrete Gronwall lemma, noting the stability estimate \eqref{equ:C infty cn}, we deduce
	\begin{align*}
		\norm{\bff{\theta}^n}{\bb{H}^1}^2 + \abs{e^n}^2
		\leq
		C(h^4+k^4) \exp\left[Ck\sum_{j=1}^{n-1} \left(1+\norm{\bff{u}_h^j}{\bb{L}^\infty}^2 \right) \right]
		\leq
		C(h^4+k^4),
	\end{align*}
	thus proving \eqref{equ:error theta n cn}. 
	
	Finally, inequality \eqref{equ:error theta Linfty cn} follows from \eqref{equ:error theta n cn} and the discrete Sobolev inequality in 2D~\eqref{equ:disc sob 2d}.
\end{proof}

We can now state the main theorem of this section which establishes optimal-order convergence in $\bb{L}^2$ and $\bb{H}^1$ norms (and in $\bb{L}^\infty$ when $d=2$).

\begin{theorem}\label{the:err cn}
	Let $\bff{u}_h^n$ and $\bff{u}$ be the solution of Algorithm~\ref{alg:bdf} and equation \eqref{equ:llb sav a}, respectively. For sufficiently small $h,k>0$, and $n=1,2,\ldots,\lfloor T/k \rfloor$,
	\begin{align}\label{equ:error Hs cn}
		\norm{\bff{u}_h^n- \bff{u}(t_n)}{\bb{H}^s}
		\leq
		C(h^{2-s}+k^2),\quad s\in \{0,1\}.
	\end{align}
	If $\mathscr{D}\subset \bb{R}^2$, then
	\begin{align}\label{equ:error Linf cn}
		\norm{\bff{u}_h^n- \bff{u}(t_n)}{\bb{L}^\infty}
		\leq
		C(h^2+k^2) \abs{\ln h}^{\frac12}.
	\end{align}
	Furthermore, the energy functional $\widehat{\mathcal{E}}[\bff{u}_h^n]$ is a good approximation of $\mathcal{E}[\bff{u}(t_n)]$ in the sense that
	\begin{align}\label{equ:good ener cn}
		\abs{\widehat{\mathcal{E}}[\bff{u}_h^n]- \mathcal{E}[\bff{u}(t_n)]} \leq C(h+k^2).
	\end{align}
	The constant $C$ in \eqref{equ:error Hs cn}, \eqref{equ:error Linf cn}, and \eqref{equ:good ener cn} {depends on $T$ and the exact solution $\bff{u}$}, but is independent of $n$, $h$, and $k$.
\end{theorem}

\begin{proof}
	The estimate \eqref{equ:error Hs cn} follows from \eqref{equ:split un cn}, \eqref{equ:error theta n cn}, and \eqref{equ:Ritz ineq}. Correspondingly, inequality \eqref{equ:error Linf cn} follows from \eqref{equ:split un cn}, \eqref{equ:error theta Linfty cn}, and \eqref{equ:Ritz infty ineq}. The proof of \eqref{equ:good ener cn} follows the same argument as in that of \eqref{equ:good ener eul}, noting that we have established \eqref{equ:error Hs cn}.
\end{proof}

\begin{remark}
	{The analysis in this paper is carried out for the quartic potential $F(\bff{u})=\frac{\kappa}{4}(|\bff{u}|^4+1)$, which is standard in the derivation of the Landau--Lifshitz--Bloch equation from Ginzburg--Landau-type free energies above the Curie temperature~\cite{ChuNie20, ChuNowChaGar06}. Higher-order polynomial potentials may also arise in Landau expansions when additional terms are retained to more accurately describe the temperature dependence of the free energy near or above the Curie point.
	We note that the proposed SAV-based schemes and their discrete energy dissipation property extend to more general polynomial-type potentials without essential modification, as the SAV reformulation in Section~\ref{sec:sav reform} remains applicable.
	
	However, the error analysis relies on the specific polynomial structure of $\mathcal{F}$ and certain $\bb{L}^\infty$-type bounds. For higher-order polynomial potentials, estimates such as those in Lemmas~\ref{lem:nonlinear euler} and~\ref{lem:nonlinear bdf} will involve $\norm{\bff{u}_h^{n-1}}{\bb{L}^\infty}$, and the available $\ell^2(0,T;\bb{L}^\infty)$ stability is not sufficient to control these terms for more general nonlinearities, especially in three dimensions (see, e.g., the treatment of $I_8$ in \eqref{equ:ineq theta e} and $J_8$ in \eqref{equ:ineq theta e cn}).
	For this reason, we restrict the rigorous error analysis to the quartic potential. Extending the analysis to more general potentials remains an interesting topic for future work.}
\end{remark}

\section{Numerical experiments}\label{sec:exp}

The results of some numerical experiments using the open-source package \textsc{FEniCS} are reported in this section. In Simulations 1--2, we fix the domain $\mathscr{D}= [-1,1]^2 \subset \bb{R}^2$. Since the exact solution is not known, we use extrapolation to verify the order of convergence experimentally. To this end, let $\bff{u}_h^n$ be the finite element solution with spatial step size $h$ and time step size $k=\lfloor T/n\rfloor$. For $s\in \{0,1\}$, define the extrapolated order of convergence
\begin{equation*}
	\text{rate}_s :=  \log_2 \left[\frac{\max_n \norm{\bff{e}_{2h}}{\bb{H}^s}}{\max_n \norm{\bff{e}_{h}}{\bb{H}^s}}\right],
\end{equation*}
where $\bff{e}_h := \bff{u}_{h}^n-\bff{u}_{h/2}^n$.

We expect that for both schemes, when $k$ is sufficiently small,~$\text{rate}_s \approx h^{2-s}$. To verify experimentally the temporal rate of convergence for scheme~\eqref{equ:fem euler}, we set $k=Ch$ in the simulation and check that $\text{rate}_1 \approx 1$ (shown by Theorem~\ref{the:err euler}). Similarly for scheme~\eqref{equ:fem cn}, we let $k=Ch$ and check that $\text{rate}_0 \approx 2$ (in light of Theorem~\ref{the:err cn}). In each experiment, we plot the graph of energy vs time to show the energy evolution of the system for both numerical schemes.

{In Simulation 3, we compare the convergence behaviour of the schemes in~\cite{Soe25} with those proposed in this paper on the computational domain $\mathscr{D} = [0,1]^2$. For each scheme, a reference solution is first computed on a sufficiently fine mesh and time step, and used as a surrogate exact solution. Spatial (resp. temporal) convergence rates are obtained by fixing the time step (resp. mesh) and measuring the error under successive uniform refinement of the mesh size $h$ (resp. time step $k$). CPU times required to run the schemes are also compared.

Finally, in Simulation 4, we investigate the behaviour of the proposed SAV-based schemes under an adaptive time-stepping strategy driven by the discrepancy between the modified and physical energies. This experiment serves as a practical illustration of energy evolution under adaptive time-stepping. The evolution of both energies and the corresponding time-step adaptation are examined to assess the robustness and efficiency of the schemes, particularly during transient regimes.}

\subsection{Simulation 1 (Algorithm~\ref{alg:euler})}\label{sec:simulation 1}
To verify the order of convergence, we perform a numerical simulation using the first-order scheme described in Algorithm~\ref{alg:euler} for the following example. The coefficients in~\eqref{equ:llb a} are chosen as: $\gamma=50, \alpha=0.5, \sigma=0.5, \kappa=\mu=1.0$. The initial condition is given by
\[
	\bff{u}_0(x,y)= \big(\cos(2\pi y), 0, \sin(2\pi x)\big).
\]

Snapshots of the magnetic spin field $\bff{u}$ at selected times are presented in Figure~\ref{fig:snapshots field 2d}. The figure indicates the formation of domain walls, separating regions where clusters of magnetisation vectors are anti-parallel to each other. It also indicates that the magnitude of the magnetisation vectors is decaying to zero, as predicted by the theory in the regime above the Curie temperature~\cite{ChuNowChaGar06, Soe25}.

Plot of $\bff{e}_h$ against $1/h$ with $k=1\times 10^{-4}$ is shown in Figures~\ref{fig:order u 1 scheme 1}. Plot of $\bff{e}_h$ against $1/h$ with $k=h/10$ to verify the temporal error of scheme \eqref{equ:fem euler} in $\bb{H}^1$ norm is shown in Figure~\ref{fig:order u 2 scheme 1}. Graphs of energy (actual and modified) vs time with various values of $h$ and $k$ are plotted in Figures~\ref{fig:energy 1 scheme 1} and \ref{fig:energy 2 scheme 1}, showing the decay of the energies. These graphs also show that the modified energy $\widetilde{\mathcal{E}}[\bff{u}_h^n]$ is a good approximation to $\mathcal{E}[\bff{u}_h^n]$.

\begin{figure}[!htb]
	\centering
	\begin{subfigure}[b]{0.27\textwidth}
		\centering
		\includegraphics[width=\textwidth]{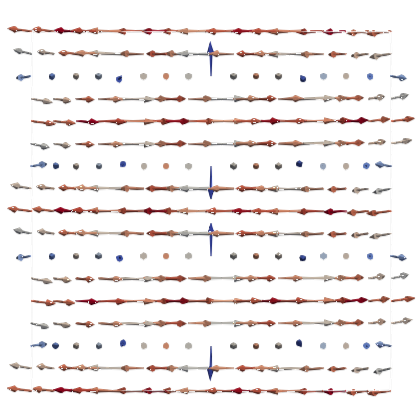}
		\caption{$t=0$}
	\end{subfigure}
	\begin{subfigure}[b]{0.27\textwidth}
		\centering
		\includegraphics[width=\textwidth]{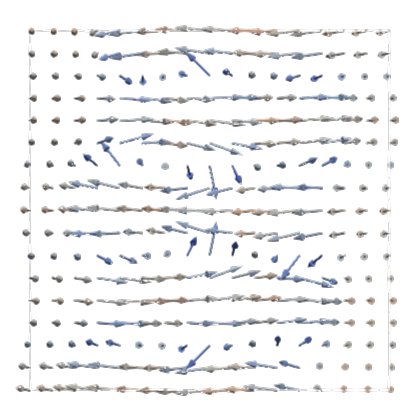}
		\caption{$t=1\times 10^{-4}$}
	\end{subfigure}
	\begin{subfigure}[b]{0.27\textwidth}
		\centering
		\includegraphics[width=\textwidth]{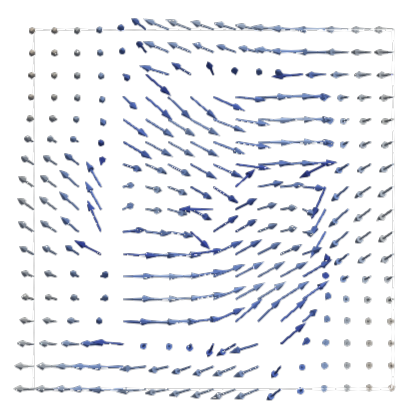}
		\caption{$t=2 \times 10^{-4}$}
	\end{subfigure}
	\begin{subfigure}[b]{0.1\textwidth}
		\centering
		\includegraphics[width=\textwidth]{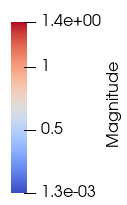}
	\end{subfigure}
	\begin{subfigure}[b]{0.27\textwidth}
		\centering
		\includegraphics[width=\textwidth]{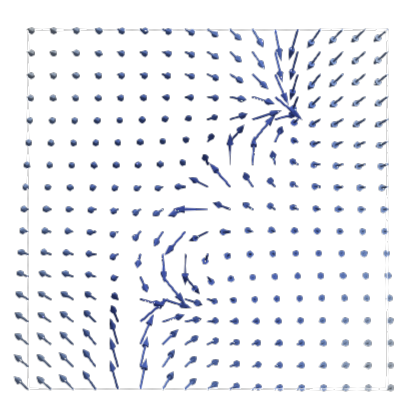}
		\caption{$t=5\times 10^{-4}$}
	\end{subfigure}
	\begin{subfigure}[b]{0.27\textwidth}
		\centering
		\includegraphics[width=\textwidth]{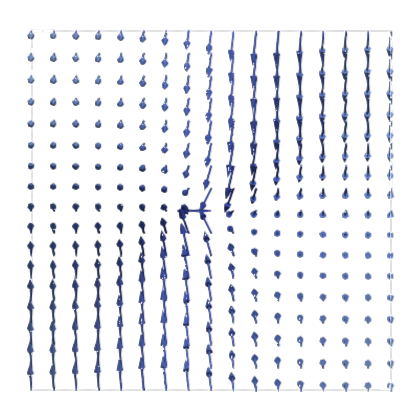}
		\caption{$t=1\times 10^{-3}$}
	\end{subfigure}
	\begin{subfigure}[b]{0.27\textwidth}
		\centering
		\includegraphics[width=\textwidth]{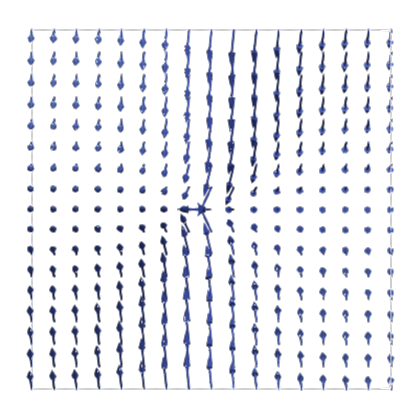}
		\caption{$t=2\times 10^{-3}$}
	\end{subfigure}
	\begin{subfigure}[b]{0.1\textwidth}
		\centering
		\includegraphics[width=\textwidth]{legend1.png}
	\end{subfigure}
	\caption{Snapshots of the spin field $\bff{u}$ (projected onto $\bb{R}^2$) for scheme~\eqref{equ:fem euler} in Simulation 1.}
	\label{fig:snapshots field 2d}
\end{figure}

\begin{figure}[!htb]
	\begin{subfigure}[b]{0.45\textwidth}
		\centering
		\begin{tikzpicture}
			\begin{axis}[
				title=Plot of $\bff{e}_h$ against $1/h$,
				height=1.05\textwidth,
				width=1\textwidth,
				xlabel= $1/h$,
				ylabel= $\bff{e}_h$,
				xmode=log,
				ymode=log,
				legend pos=south west,
				legend cell align=left,
				]
				\addplot+[mark=*,blue] coordinates {(4,7.8)(8,4.8)(16,2.8)(32,1.5)(64,0.75)(128,0.38)};
				\addplot+[mark=*,green] coordinates {(4,0.77)(8,0.35)(16,0.2)(32,0.11)(64,0.06)(128,0.02)};
				\addplot+[mark=*,red] coordinates {(4,1.2)(8,0.4)(16,0.146)(32,0.0469)(64,0.013)(128,0.0033)};
				\addplot+[dashed,no marks,blue,domain=40:130]{109/x};
				\addplot+[dashed,no marks,red,domain=40:130]{15/x^2};
				\legend{\small{$\max_n \norm{\bff{e}_h}{\bb{H}_0^1}$},
				\small{$\max_n \norm{\bff{e}_h}{\bb{L}^\infty}$},	 \small{$\max_n \norm{\bff{e}_h}{\bb{L}^2}$}, \small{order 1 line}, \small{order 2 line}}
			\end{axis}
		\end{tikzpicture}
		\caption{Spatial error order for scheme~\eqref{equ:fem euler} in Simulation 1.}
		\label{fig:order u 1 scheme 1}
	\end{subfigure}
	\hspace{1em}
	\begin{subfigure}[b]{0.45\textwidth}
		\centering
		\begin{tikzpicture}
			\begin{axis}[
				title=Plot of $\bff{e}_h$ against $1/h$,
				height=1.05\textwidth,
				width=1\textwidth,
				xlabel= $1/h$,
				ylabel= $\bff{e}_h$,
				xmode=log,
				ymode=log,
				legend pos=south west,
				legend cell align=left,
				]
				\addplot+[mark=*,red] coordinates {(8,4.8)(16,3.3)(32,2.0)(64,1.3)};
				\addplot+[dashed,no marks,red,domain=20:60]{47/x};
				\legend{\small{$\max_n \norm{\bff{e}_h}{\bb{H}_0^1}$},  \small{order 1 line}}
			\end{axis}
		\end{tikzpicture}
		\caption{Error order for scheme~\eqref{equ:fem euler} in Simulation 1, taking $k=Ch$.}
		\label{fig:order u 2 scheme 1}
	\end{subfigure}
	\caption{Spatial and temporal order of convergence for Simulation 1.}
\end{figure}

\begin{figure}[!htb]
	\begin{subfigure}[b]{0.48\textwidth}
		\centering
		\includegraphics[width=\textwidth]{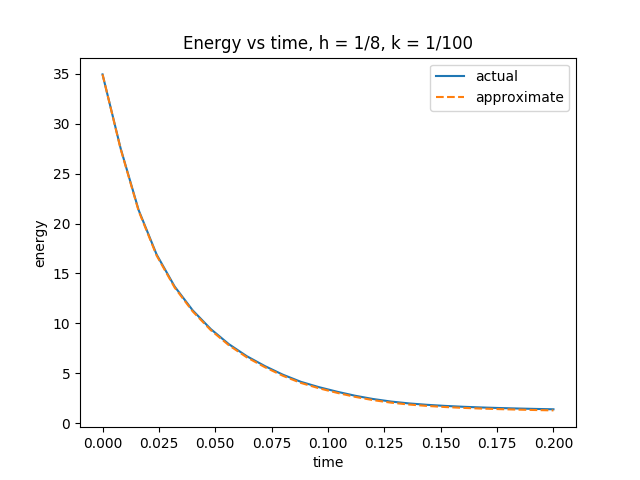}
		\caption{$h=1/8$ and $k=1/100$.}
		\label{fig:energy 1 scheme 1}
	\end{subfigure}
	\hspace{1em}
	\begin{subfigure}[b]{0.48\textwidth}
		\centering
		\includegraphics[width=\textwidth]{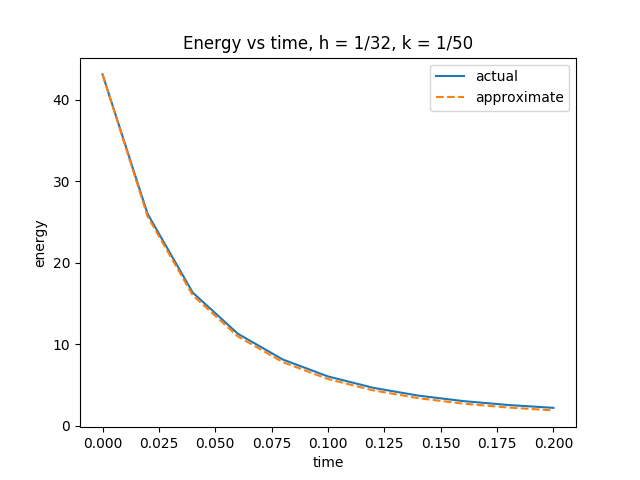}
		\caption{$h=1/32$ and $k=1/50$.}
		\label{fig:energy 2 scheme 1}
	\end{subfigure}
	\caption{Plots of (approximate and actual) energy vs time for scheme~\eqref{equ:fem euler} in Simulation~1.}
\end{figure}

\subsection{Simulation 2 (Algorithm~\ref{alg:bdf})}\label{sec:simulation 2}
We perform a numerical simulation using the second-order scheme in Algorithm~\ref{alg:bdf} with the following academic example. The coefficients in~\eqref{equ:llb a} are set to be: $\gamma=100, \alpha=0.1, \sigma=0.1, \kappa=2.0, \mu=1.0$. The initial data is taken to be
\[
\bff{u}_0(x,y)= \big(-y, x, \cos(2\pi x)\big).
\]

Snapshots of the magnetic spin field $\bff{u}$ at selected times are presented in Figure~\ref{fig:snapshots field 2d 2}.
Plot of $\bff{e}_h$ against $1/h$ with $k=1\times 10^{-5}$ is shown in Figures~\ref{fig:order u 1 scheme 2}. Plot of $\bff{e}_h$ against $1/h$ with $k=h/20$ to verify the temporal error of scheme \eqref{equ:fem cn} in $\bb{H}^1$ norm is shown in Figure~\ref{fig:order u 2 scheme 2}. Graphs of energy (actual and modified) vs time with various values of $h$ and $k$ are plotted in Figures~\ref{fig:energy 1 scheme 2} and \ref{fig:energy 2 scheme 2}, showing the decay of the energies. These graphs also show that the modified energy $\widehat{\mathcal{E}}[\bff{u}_h^n]$ is a good approximation to $\mathcal{E}[\bff{u}_h^n]$.

\begin{figure}[!htb]
	\centering
	\begin{subfigure}[b]{0.27\textwidth}
		\centering
		\includegraphics[width=\textwidth]{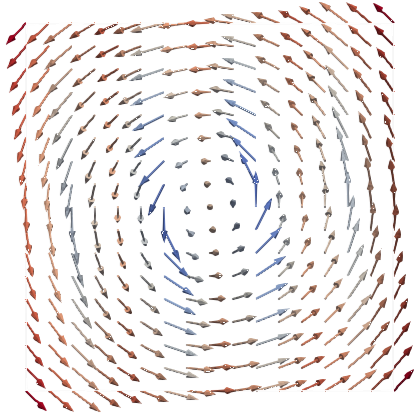}
		\caption{$t=0$}
	\end{subfigure}
	\begin{subfigure}[b]{0.27\textwidth}
		\centering
		\includegraphics[width=\textwidth]{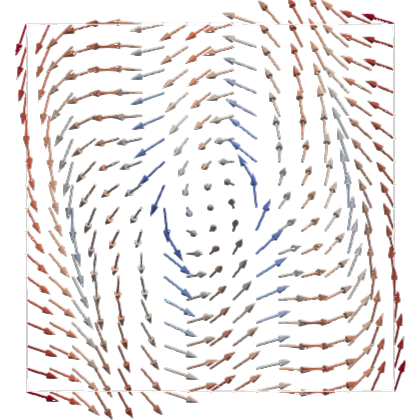}
		\caption{$t=2\times 10^{-3}$}
	\end{subfigure}
	\begin{subfigure}[b]{0.27\textwidth}
		\centering
		\includegraphics[width=\textwidth]{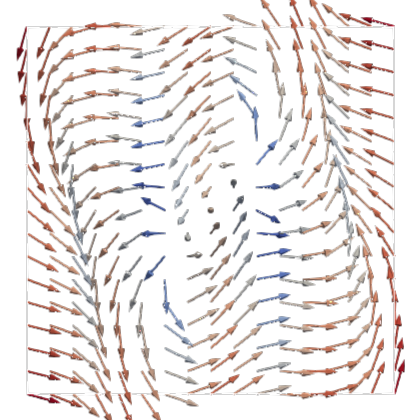}
		\caption{$t=5 \times 10^{-3}$}
	\end{subfigure}
	\begin{subfigure}[b]{0.1\textwidth}
		\centering
		\includegraphics[width=\textwidth]{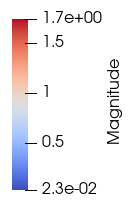}
	\end{subfigure}
	\begin{subfigure}[b]{0.27\textwidth}
		\centering
		\includegraphics[width=\textwidth]{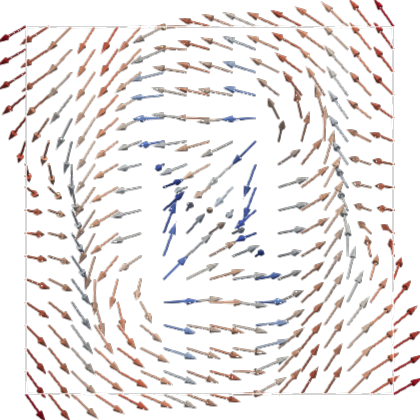}
		\caption{$t=1\times 10^{-2}$}
	\end{subfigure}
	\begin{subfigure}[b]{0.27\textwidth}
		\centering
		\includegraphics[width=\textwidth]{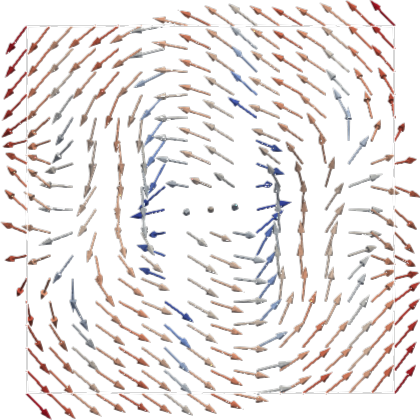}
		\caption{$t=1.5\times 10^{-2}$}
	\end{subfigure}
	\begin{subfigure}[b]{0.27\textwidth}
		\centering
		\includegraphics[width=\textwidth]{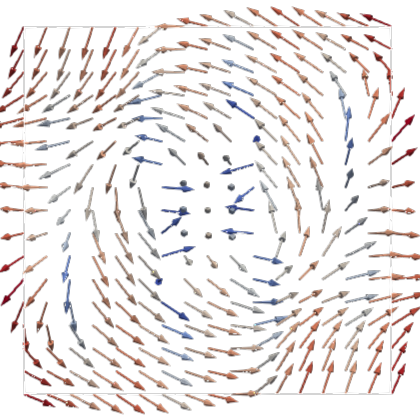}
		\caption{$t=2\times 10^{-2}$}
	\end{subfigure}
	\begin{subfigure}[b]{0.1\textwidth}
		\centering
		\includegraphics[width=\textwidth]{legend_2.png}
	\end{subfigure}
	\caption{Snapshots of the spin field $\bff{u}$ (projected onto $\bb{R}^2$) for scheme~\eqref{equ:fem cn} in Simulation 2.}
	\label{fig:snapshots field 2d 2}
\end{figure}

\begin{figure}[!htb]
	\begin{subfigure}[b]{0.45\textwidth}
		\centering
		\begin{tikzpicture}
			\begin{axis}[
				title=Plot of $\bff{e}_h$ against $1/h$,
				height=1.3\textwidth,
				width=1\textwidth,
				xlabel= $1/h$,
				ylabel= $\bff{e}_h$,
				xmode=log,
				ymode=log,
				legend pos=south west,
				legend cell align=left,
				]
				\addplot+[mark=*,blue] coordinates {(8,3.87)(16,2.0)(32,1.03)(64,0.516)(128,0.26)};
				\addplot+[mark=*,green] coordinates {(8,1.25)(16,0.2)(32,0.077)(64,0.025)(128,0.007)};
				\addplot+[mark=*,red] coordinates {(8,0.25)(16,0.066)(32,0.017)(64,0.0043)(128,0.0011)};
				\addplot+[dashed,no marks,blue,domain=40:130]{80/x};
				\addplot+[dashed,no marks,red,domain=40:130]{39/x^2};
				\legend{\small{$\max_n \norm{\bff{e}_h}{\bb{H}_0^1}$},
					\small{$\max_n \norm{\bff{e}_h}{\bb{L}^\infty}$},	 \small{$\max_n \norm{\bff{e}_h}{\bb{L}^2}$}, \small{order 1 line}, \small{order 2 line}}
			\end{axis}
		\end{tikzpicture}
		\caption{Spatial error order for scheme~\eqref{equ:fem cn} in Simulation 2.}
		\label{fig:order u 1 scheme 2}
	\end{subfigure}
	\hspace{1em}
	\begin{subfigure}[b]{0.45\textwidth}
		\centering
		\begin{tikzpicture}
			\begin{axis}[
				title=Plot of $\bff{e}_h$ against $1/h$,
				height=1.3\textwidth,
				width=1\textwidth,
				xlabel= $1/h$,
				ylabel= $\bff{e}_h$,
				xmode=log,
				ymode=log,
				legend pos=south west,
				legend cell align=left,
				]
				\addplot+[mark=*,red] coordinates {(8,14)(16,2.3)(32,0.75)(64,0.19)};
				\addplot+[dashed,no marks,red,domain=20:60]{280/x^2};
				\legend{\small{$\max_n \norm{\bff{e}_h}{\bb{L}^2}$},  \small{order 2 line}}
			\end{axis}
		\end{tikzpicture}
		\caption{Error order for scheme~\eqref{equ:fem cn} in Simulation 2, taking $k=Ch$.}
		\label{fig:order u 2 scheme 2}
	\end{subfigure}
	\caption{Spatial and temporal order of convergence for Simulation 2.}
\end{figure}

\begin{figure}[!htb]
	\begin{subfigure}[b]{0.48\textwidth}
		\centering
		\includegraphics[width=\textwidth]{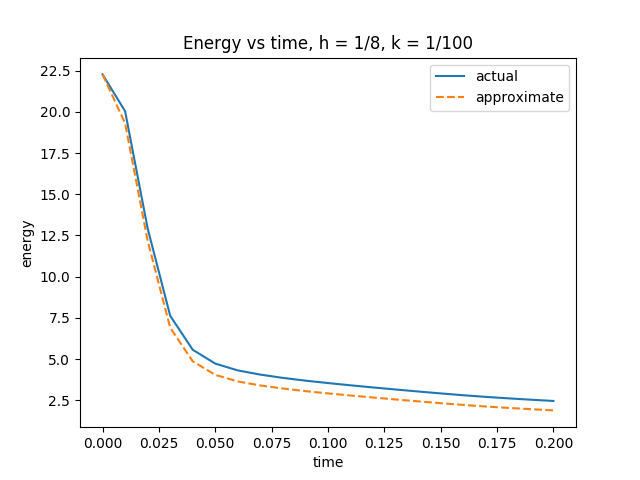}
		\caption{$h=1/8$ and $k=1/100$.}
		\label{fig:energy 1 scheme 2}
	\end{subfigure}
	\hspace{1em}
	\begin{subfigure}[b]{0.48\textwidth}
		\centering
		\includegraphics[width=\textwidth]{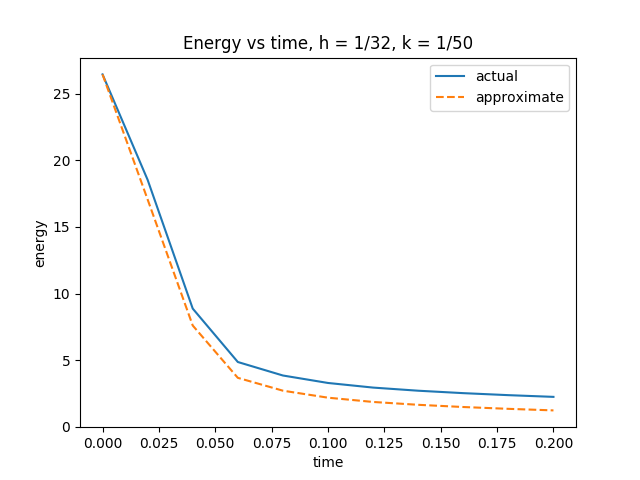}
		\caption{$h=1/32$ and $k=1/50$.}
		\label{fig:energy 2 scheme 2}
	\end{subfigure}
	\caption{Plots of (approximate and actual) energy vs time for scheme~\eqref{equ:fem cn} in Simulation~2.}
\end{figure}

\subsection{Simulation 3 (Comparison with other schemes)}\label{sec:simulation 3}

{Two finite element schemes were proposed in~\cite{Soe25}, which are stated below for the convenience of the readers. The first scheme is a linear semi-implicit scheme, which is not necessarily energy-stable:
	
\begin{algorithm}[Semi-implicit linear non-SAV FEM~\cite{Soe25}]\label{alg:linear nonsav}
	Let $h>0$ and $k>0$ be given.
	\\
	\textbf{Input}: Given $\bff{u}_h^0\in \bb{V}_h$.
	\\
	\textbf{For} $n=1$ to $N$, where $N=\lfloor T/k \rfloor$, compute $\bff{u}_h^n\in \bb{V}_h$ satisfying for any $\bff{\phi}\in \bb{V}_h$,
	\begin{align*}
		\inpro{\dtt \bff{u}_h^n}{\bff{\phi}}= \gamma \sigma \inpro{\bff{u}_h^{n-1} \times \nabla \bff{u}_h^n}{\nabla\bff{\phi}} - \alpha\sigma \inpro{\nabla \bff{u}_h^n}{\nabla \bff{\phi}} - \alpha\kappa\mu \inpro{\bff{u}_h^n}{\bff{\phi}} - \alpha\kappa \inpro{\abs{\bff{u}_h^{n-1}}^2 \bff{u}_h^n}{\bff{\phi}}.
	\end{align*}
	\textbf{Output}: a sequence of discrete functions $\{\bff{u}_h^n\}_{1\leq n\leq N}$.
\end{algorithm}

The second scheme is a nonlinear scheme, which satisfies a discrete energy law:

\begin{algorithm}[Nonlinear non-SAV FEM~\cite{Soe25}]\label{alg:nonlinear nonsav}
	Let $h>0$ and $k>0$ be given.
	\\
	\textbf{Input}: Given $\bff{u}_h^0\in \bb{V}_h$.
	\\
	\textbf{For} $n=1$ to $N$, where $N=\lfloor T/k \rfloor$, compute $\bff{u}_h^n\in \bb{V}_h$ satisfying for any $\bff{\phi}\in \bb{V}_h$,
	\begin{align*}
		\inpro{\dtt \bff{u}_h^n}{\bff{\phi}}= -\gamma \inpro{\bff{u}_h^n \times \bff{H}_h^n}{\bff{\phi}} + \alpha \inpro{\bff{H}_h^n}{\bff{\phi}},
	\end{align*}
	with $\bff{H}_h^n:= \sigma \Delta_h \bff{u}_h^n- \kappa\mu \bff{u}_h^n- \kappa P_h\big(\abs{\bff{u}_h^n}^2 \bff{u}_h^n\big) \in \bb{V}_h$. Here, $\Delta_h$ is the discrete Laplacian and $P_h$ is the $\bb{L}^2$-orthogonal projector onto $\bb{V}_h$.
	\\
	\textbf{Output}: a sequence of discrete functions $\{\bff{u}_h^n\}_{1\leq n\leq N}$.
\end{algorithm}
	
We compare the above schemes with the SAV-based schemes introduced in this paper.
The computational domain is $\mathscr{D} = [0,1]^2$, and the coefficients in~\eqref{equ:llb a} are chosen as $\alpha=\sigma=0.1$ and $\gamma=\kappa=\mu=1.0$. The initial condition is given by
\[
\bff{u}_0(x,y)=\big(\sin(\pi x)\cos(\pi y), -\cos(\pi x) \sin(\pi y), \cos(\pi x)\cos(\pi y)\big).
\]
All simulations are performed up to the final time $T$, to be specified below.

To assess convergence, we take $T=0.05$ and compute a reference solution $\bff{u}_{\mathrm{ref}}(T)$ on a sufficiently fine discretisation with mesh size $h=1/80$ and time step $k=2\times 10^{-4}$, which is treated as a surrogate exact solution. Let $\bff{u}_h^N$ denote the numerical solution at time $T$, and define the error $\widetilde{\bff{e}}_h^N:= \bff{u}_h^N- \bff{u}_{\mathrm{ref}}(T)$.
Spatial convergence is examined by fixing the time step $k$ and computing the $\bb{L}^2$ norm of $\widetilde{\bff{e}}_h^N$ on a sequence of successively refined meshes obtained by halving $h$; the results are reported in Table~\ref{tab:temp}. Temporal convergence is assessed analogously by fixing $h$ and successively halving $k$, with results shown in Table~\ref{tab:spat}. The computational cost of each scheme, measured in CPU time, is reported in Tables~\ref{tab:cpu temp} and~\ref{tab:cpu spat}.

The observed convergence rates are consistent with the theoretical predictions for all schemes. In terms of computational efficiency, the linear scheme in~\cite{Soe25} is the least expensive, as it requires only the solution of linear systems at each time step. In contrast, the nonlinear scheme in~\cite{Soe25} incurs the highest computational cost (in terms of CPU time) due to the need to solve nonlinear systems iteratively. The SAV-based schemes proposed in this paper retain linearity (at the expense of introducing an auxiliary variable) while provably enforcing a discrete energy law, and thus provide a favourable compromise between computational efficiency and energy stability, without loss of accuracy or convergence order.

\begin{table}[!h]
	\centering
	\small
	\setlength{\tabcolsep}{6pt}
	\begin{tabular}{c cc cc cc cc}
		\toprule
		& \multicolumn{2}{c}{linear non-SAV}
		& \multicolumn{2}{c}{nonlinear non-SAV}
		& \multicolumn{2}{c}{Euler--SAV}
		& \multicolumn{2}{c}{BDF2--SAV} \\
		\cmidrule(lr){2-3} \cmidrule(lr){4-5}
		\cmidrule(lr){6-7} \cmidrule(lr){8-9}
		$k$ & error & rate & error & rate & error & rate & error & rate \\
		\midrule
		$1.0\times 10^{-2}$
		& $1.048\times 10^{-2}$ & --
		& $1.018\times 10^{-2}$ & --
		& $1.032\times 10^{-2}$ & --
		& $7.013\times 10^{-3}$ & -- \\
		
		$5.0\times 10^{-3}$
		& $6.704\times 10^{-3}$ & 0.64
		& $6.512\times 10^{-3}$ & 0.64
		& $6.596\times 10^{-3}$ & 0.65
		& $3.736\times 10^{-3}$ & 0.91 \\
		
		$2.5\times 10^{-3}$
		& $4.026\times 10^{-3}$ & 0.74
		& $3.916\times 10^{-3}$ & 0.73
		& $3.964\times 10^{-3}$ & 0.73
		& $1.744\times 10^{-3}$ & 1.10 \\
		
		$1.25\times 10^{-3}$
		& $2.203\times 10^{-3}$ & 0.87
		& $2.146\times 10^{-3}$ & 0.87
		& $2.171\times 10^{-3}$ & 0.87
		& $6.677\times 10^{-4}$ & 1.39 \\
		
		$6.25\times 10^{-4}$
		& $1.023\times 10^{-3}$ & 1.11
		& $9.972\times 10^{-4}$ & 1.11
		& $1.009\times 10^{-3}$ & 1.11
		& $1.955\times 10^{-4}$ & 1.78 \\
		\bottomrule
	\end{tabular}
	\caption{$\bb{L}^2$ errors, obtained by computing the $\bb{L}^2$ norms of $\widetilde{\bff{e}}_h^N:= \bff{u}_h^N- \bff{u}_{\mathrm{ref}}(T)$, and experimental temporal convergence rates at $T=0.05$ for successively halved time step sizes $k$. The results are shown for the linear and nonlinear non-SAV schemes (Algorithms~\ref{alg:linear nonsav} and~\ref{alg:nonlinear nonsav}) from~\cite{Soe25}, as well as the Euler--SAV scheme (Algorithm~\ref{alg:euler}) and the BDF2--SAV scheme (Algorithm~\ref{alg:bdf}). The observed temporal convergence rates are consistent with the theory.}
	\label{tab:temp}
\end{table}

\begin{table}[!h]
	\centering
	\small
	\setlength{\tabcolsep}{6pt}
	\begin{tabular}{c cc cc cc cc}
		\toprule
		& \multicolumn{2}{c}{linear non-SAV}
		& \multicolumn{2}{c}{nonlinear non-SAV}
		& \multicolumn{2}{c}{Euler--SAV}
		& \multicolumn{2}{c}{BDF2--SAV} \\
		\cmidrule(lr){2-3} \cmidrule(lr){4-5}
		\cmidrule(lr){6-7} \cmidrule(lr){8-9}
		$h$ & error & rate & error & rate & error & rate & error & rate \\
		\midrule
		$1/4$
		& $1.132\times 10^{-1}$ & --
		& $1.132\times 10^{-1}$ & --
		& $7.100\times 10^{-2}$ & --
		& $7.114\times 10^{-2}$ & -- \\
		
		$1/8$
		& $3.569\times 10^{-2}$ & 1.67
		& $3.571\times 10^{-2}$ & 1.66
		& $2.787\times 10^{-2}$ & 1.35
		& $2.817\times 10^{-2}$ & 1.34 \\
		
		$1/16$
		& $1.298\times 10^{-2}$ & 1.46
		& $1.299\times 10^{-2}$ & 1.46
		& $1.190\times 10^{-2}$ & 1.23
		& $1.253\times 10^{-2}$ & 1.17 \\
		
		$1/32$
		& $3.815\times 10^{-3}$ & 1.77
		& $3.822\times 10^{-3}$ & 1.76
		& $3.664\times 10^{-3}$ & 1.70
		& $4.096\times 10^{-3}$ & 1.61 \\
		
		$1/64$
		& $5.238\times 10^{-4}$ & 2.85
		& $5.247\times 10^{-4}$ & 2.85
		& $5.121\times 10^{-4}$ & 2.83
		& $5.972\times 10^{-4}$ & 2.77 \\
		\bottomrule
	\end{tabular}
	\caption{$\bb{L}^2$ errors, obtained by computing the $\bb{L}^2$ norms of $\widetilde{\bff{e}}_h^N:= \bff{u}_h^N- \bff{u}_{\mathrm{ref}}(T)$, and experimental spatial convergence rates at $T=0.05$ for successively halved mesh sizes $h$. The results are shown for the linear and nonlinear non-SAV schemes (Algorithms~\ref{alg:linear nonsav} and~\ref{alg:nonlinear nonsav}) from~\cite{Soe25}, as well as the Euler--SAV scheme (Algorithm~\ref{alg:euler}) and the BDF2--SAV scheme (Algorithm~\ref{alg:bdf}). The observed spatial convergence rates are consistent with the theory.}
	\label{tab:spat}
\end{table}

\begin{table}[!h]
	\centering
	\small
	\setlength{\tabcolsep}{7pt}
	\begin{tabular}{c cccc}
		\toprule
		& \multicolumn{4}{c}{CPU time (s)} \\
		\cmidrule(lr){2-5}
		$k$ & linear non-SAV & nonlinear non-SAV & Euler--SAV & BDF2--SAV \\
		\midrule
		$1.0\times 10^{-2}$ & $0.420$ & $1.103$ & $0.517$ & $0.586$ \\
		$5.0\times 10^{-3}$ & $0.846$ & $1.850$ & $1.034$ & $1.257$ \\
		$2.5\times 10^{-3}$ & $1.675$ & $3.234$ & $2.210$ & $2.284$ \\
		$1.25\times 10^{-3}$ & $3.532$ & $5.816$ & $4.542$ & $4.609$ \\
		$6.25\times 10^{-4}$ & $6.773$ & $10.875$ & $8.933$ & $8.949$ \\
		\bottomrule
	\end{tabular}
	\caption{CPU times (in seconds) at $T=0.05$ for successively refined time step sizes $k$. Results are reported for the linear and nonlinear non-SAV schemes (Algorithms~\ref{alg:linear nonsav} and~\ref{alg:nonlinear nonsav}) from~\cite{Soe25}, and for the Euler--SAV (Algorithm~\ref{alg:euler}) and BDF2--SAV (Algorithm~\ref{alg:bdf}) schemes.}
	\label{tab:cpu temp}
\end{table}

\begin{table}[!h]
	\centering
	\small
	\setlength{\tabcolsep}{7pt}
	\begin{tabular}{c cccc}
		\toprule
		& \multicolumn{4}{c}{CPU time (s)} \\
		\cmidrule(lr){2-5}
		$h$ & linear non-SAV & nonlinear non-SAV & Euler--SAV & BDF2--SAV \\
		\midrule
		$1/4$  & $0.093$  & $0.115$  & $0.296$  & $0.357$ \\
		$1/8$  & $0.192$  & $0.256$  & $0.479$  & $0.489$ \\
		$1/16$ & $0.556$  & $0.878$  & $0.849$  & $0.914$ \\
		$1/32$ & $2.119$  & $3.671$  & $2.931$  & $3.327$ \\
		$1/64$ & $12.280$ & $19.078$ & $15.854$ & $15.807$ \\
		\bottomrule
	\end{tabular}
	\caption{CPU times (in seconds) at $T=0.05$ for successively refined mesh sizes $h$. Results are reported for the linear and nonlinear non-SAV schemes (Algorithms~\ref{alg:linear nonsav} and~\ref{alg:nonlinear nonsav}) from~\cite{Soe25}, and for the Euler--SAV (Algorithm~\ref{alg:euler}) and BDF2--SAV (Algorithm~\ref{alg:bdf}) schemes.}
	\label{tab:cpu spat}
\end{table}

\subsection{Simulation 4 (Adaptive time-stepping and energy dynamics)}\label{sec:simulation 4}

While the theoretical analysis establishes that the proposed SAV-based schemes unconditionally dissipate the modified discrete energy for any fixed time step $k$, this property alone does not guarantee that the modified energy remains a tight approximation of the actual physical energy $\mathcal{E}[\bff{u}]$ as time progresses, especially during rapid phase transitions. To demonstrate the practical robustness of the schemes, we implement a local phase-error-driven adaptive time-stepping strategy in this simulation.

We define the SAV discrepancy at time level $t_n$ as $\delta^n:= (r_h^n)^2- \mathcal{F}[\bff{u}_h^n]$. To prevent the artificial accumulation of global historical error from locking the time step size, we control the local phase error increment injected during the trial step. The relative local phase error is defined as
\begin{align*}
	e_{\mathrm{loc}}^n := \frac{\abs{\delta^{n+1}-\delta^n}}{\mathcal{E}[\bff{u}_h^{n+1}]}.
\end{align*}
Given a prescribed tolerance $\mathsf{tol}$, the new time step size is computed via a continuous fractional controller:
\begin{align*}
	k_{\mathrm{new}} = \rho \left( \frac{\mathsf{tol}}{e_{\mathrm{loc}}^n} \right)^{1/p} k_{\mathrm{old}},
\end{align*}
where $p$ is the temporal convergence order of the scheme ($p=1$ for Euler--SAV and $p=2$ for BDF2--SAV) and $\rho = 0.9$ is a safety factor designed to absorb higher-order nonlinear fluctuations and minimize step rejections. To prevent step-size chattering and avoid unnecessary matrix refactorisation, a dead-band filter is applied: $k$ is held strictly constant if the proposed scaling factor falls within $[0.9, 1.1]$. If $e_{\mathrm{loc}}^n > \mathsf{tol}$, the step is rejected and recomputed with $k_{\mathrm{new}}$; otherwise, the step is accepted and $k_{\mathrm{new}}$ is utilised for the subsequent step.

To evaluate this adaptive mechanism, we set the computational domain $\mathscr{D}=[0,1]^2$. The physical coefficients are chosen as $\gamma=50.0$, $\alpha=\sigma=\mu=1.0$, and $\kappa=5.0$. We impose the initial condition
\[
	\bff{u}_0(x,y)= \big( \cos(\pi x), \cos(2\pi y), \cos(\pi x)\cos(\pi y)\big).
\]
We run the simulation up to the final time $T=0.005$ with an initial time step $k=10^{-5}$ and a tolerance $\mathsf{tol}=10^{-5}$.

Figures~\ref{fig:energy adapt 1} and \ref{fig:energy adapt 2} display the evolution of the actual and modified energies alongside the dynamically adapted time step $k$ for both SAV schemes. In both cases, the adaptive controller successfully restricts the local phase error, ensuring the modified energy $\widetilde{\mathcal{E}}$ tightly anchors the actual physical energy $\mathcal{E}$ throughout the simulation. As the initial stiff transients decay, the controller scales up the time step. Notably, the plots highlight a clear efficiency advantage for the BDF2--SAV scheme: exploiting its $\mathcal{O}(k^2)$ temporal convergence, it aggressively expands the time step to roughly $7 \times 10^{-5}$, whereas the first-order Euler--SAV scheme remains constrained below $2 \times 10^{-5}$ to maintain the same prescribed tolerance.

\begin{figure}[!htb]
	\centering
	\includegraphics[width=0.8\textwidth]{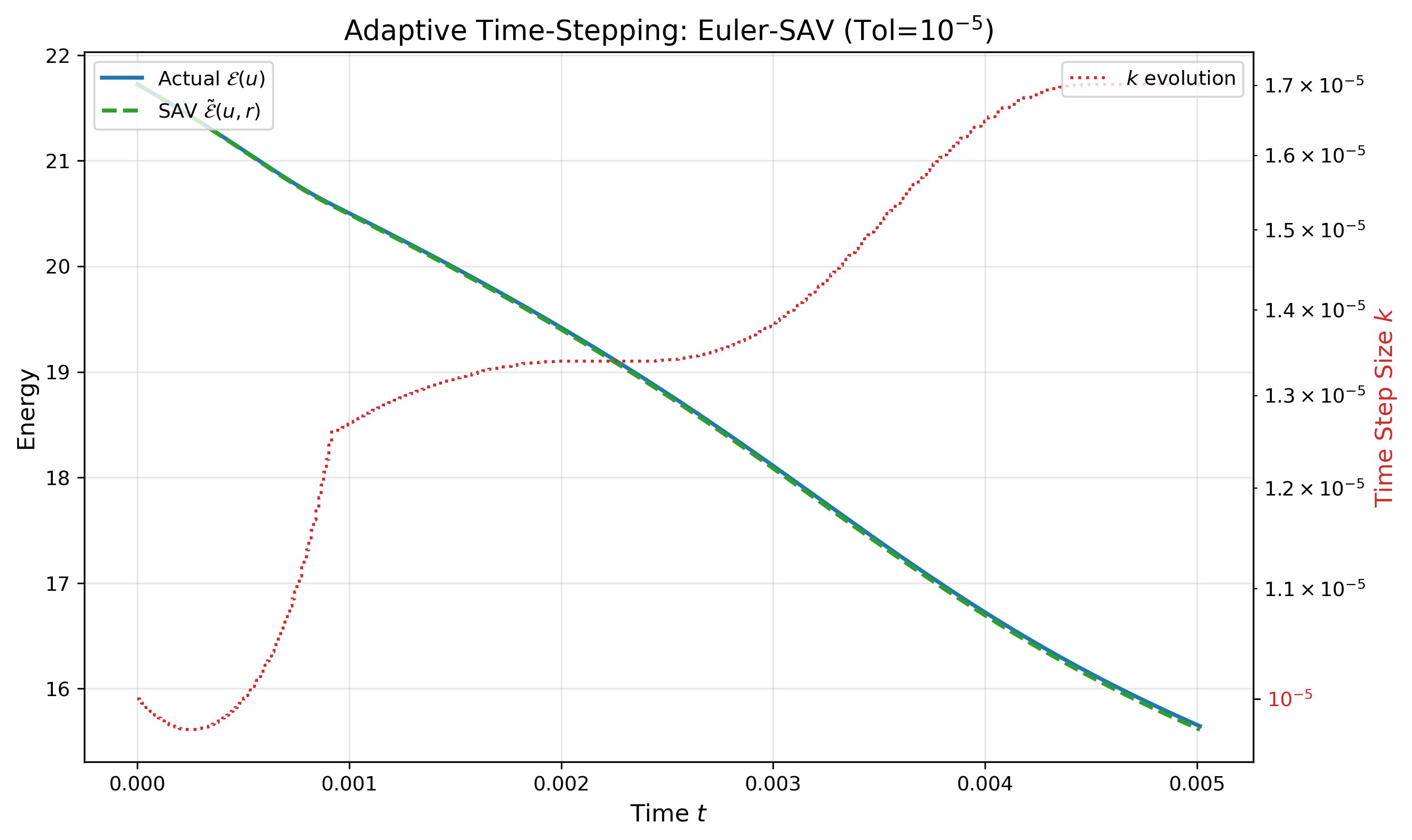}
	\caption{Evolution of the actual physical energy and modified SAV energy alongside the dynamically selected time step size $k$ (right axis) for the first-order Euler--SAV scheme (Algorithm~\ref{alg:euler}).}
	\label{fig:energy adapt 1}
\end{figure}

\begin{figure}[!htb]
	\centering
	\includegraphics[width=0.8\textwidth]{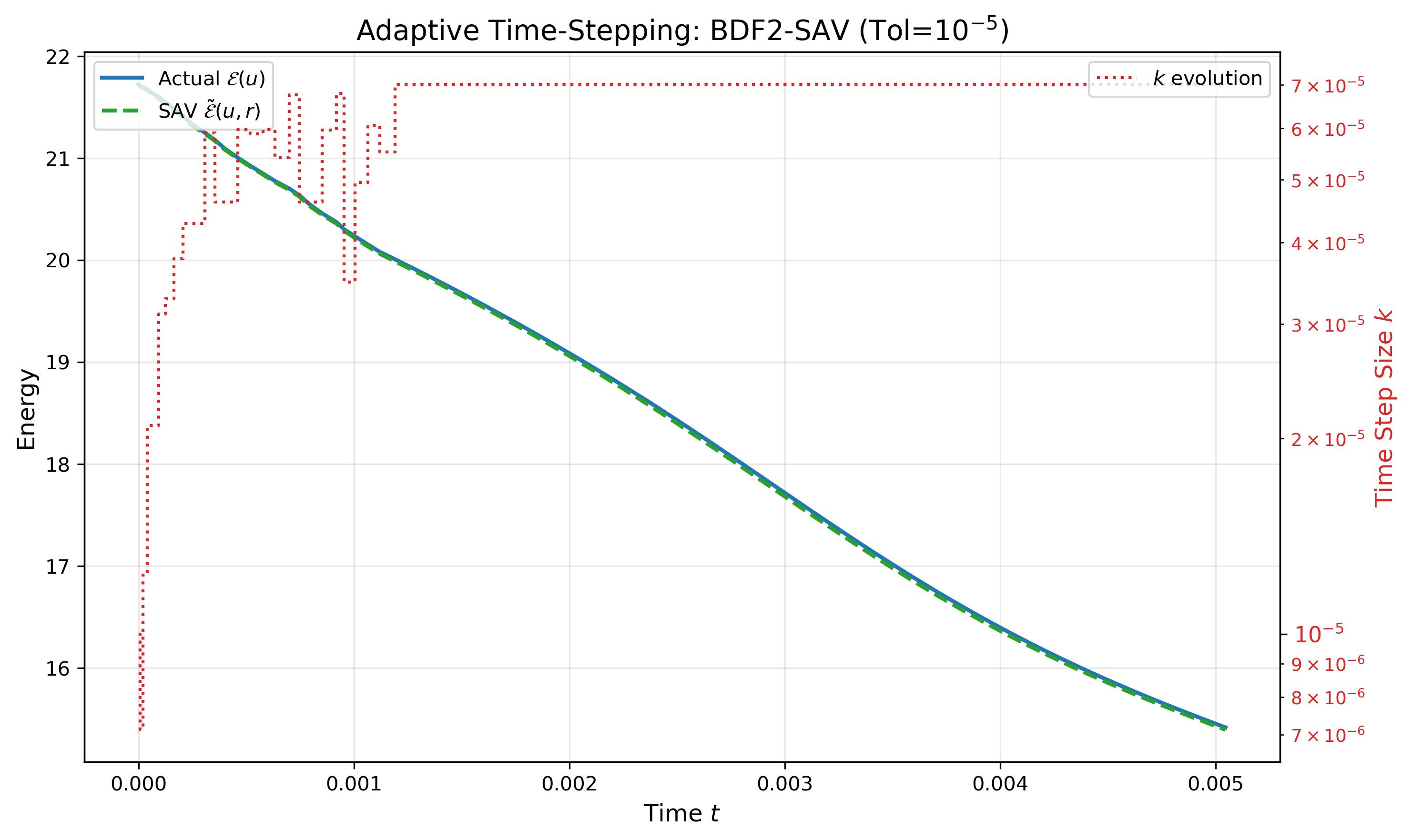}
	\caption{Evolution of the actual physical energy and modified SAV energy alongside the dynamically selected time step size $k$ (right axis) for the first-order BDF2--SAV scheme (Algorithm~\ref{alg:bdf}).}
	\label{fig:energy adapt 2}
\end{figure}
}

\section*{Statements and declarations}

\subsection*{Conflict of interest}
The author has no competing interests to declare that are relevant to the content of this article.

\subsection*{Funding statement}
This research is supported by the Commonwealth through an Australian Government Research Training Program Scholarship [DOI: \href{https://doi.org/10.82133/C42F-K220}{https://doi.org/10.82133/C42F-K220}]. Financial support from the Edinburgh Mathematical Society Research Fund (EMS-RSF) Grant E2503-LIN and from the Australian Research Council (Grant DP220101811) is gratefully acknowledged.

\subsection*{Acknowledgement}
Part of this work was carried out during a research visit to the Department of Mathematics at the University of Dundee. The author is grateful to Prof. Ping Lin for his generous hospitality and to the university for providing a stimulating research environment. The author also thanks Prof. Thanh Tran for valuable suggestions. Finally, the author thanks the referees for their constructive comments, which led to a significant improvement of the paper.

%\bibliographystyle{myabbrv}
%\bibliography{mybib}

\begin{thebibliography}{10}
	
	\bibitem{AtxHinNow16}
	U.~Atxitia, D.~Hinzke, and U.~Nowak.
	\newblock Fundamentals and applications of the {L}andau–{L}ifshitz–{B}loch
	equation.
	\newblock {\em Journal of Physics D: Applied Physics},  {\bf 50} (2016),
	033003.
	
	\bibitem{Bar84}
	V.~G. Baryakhtar.
	\newblock Phenomenological description of relaxation processes in magnets.
	\newblock {\em Zh. Eksp. Teor. Fiz.},  {\bf 87} (1984).
	
	\bibitem{BenEssAyo24}
	M.~Benmouane, E.-H. Essoufi, and C.~Ayouch.
	\newblock A finite element scheme for the {L}andau--{L}ifshitz--{B}loch
	equation.
	\newblock {\em Comput. Appl. Math.},  {\bf 43} (2024), Paper No. 394.
	
	\bibitem{Bre04}
	S.~C. Brenner.
	\newblock Discrete {S}obolev and {P}oincar\'e{} inequalities for piecewise
	polynomial functions.
	\newblock {\em Electron. Trans. Numer. Anal.},  {\bf 18} (2004), 42--48.
	
	\bibitem{BreSco08}
	S.~C. Brenner and L.~R. Scott.
	\newblock {\em The mathematical theory of finite element methods}, volume~15 of
	{\em Texts in Applied Mathematics}.
	\newblock Springer, New York, third edition, 2008.
	
	\bibitem{Bre11}
	H. Brezis.
	\newblock {\em Functional analysis, {S}obolev spaces and partial differential
		equations}, Universitext.
	\newblock Springer, New York, 2011.
	
	\bibitem{CheMaoShe20}
	H.~Chen, J.~Mao, and J.~Shen.
	\newblock Optimal error estimates for the scalar auxiliary variable
	finite-element schemes for gradient flows.
	\newblock {\em Numer. Math.},  {\bf 145} (2020), 167--196.
	
	\bibitem{ChuNie20}
	O.~Chubykalo-Fesenko and P.~Nieves.
	\newblock {\em Landau-Lifshitz-Bloch Approach for Magnetization Dynamics Close
		to Phase Transition}, pages 867--893.
	\newblock Springer International Publishing, Cham, 2020.
	
	\bibitem{ChuNowChaGar06}
	O.~Chubykalo-Fesenko, U.~Nowak, R.~W. Chantrell, and D.~Garanin.
	\newblock Dynamic approach for micromagnetics close to the {C}urie temperature.
	\newblock {\em Phys. Rev. B},  {\bf 74} (2006), 094436.
	
	\bibitem{CroTho87}
	M.~Crouzeix and V.~Thom\'{e}e.
	\newblock The stability in {$L_p$} and {$W^1_p$} of the {$L_2$}-projection onto
	finite element function spaces.
	\newblock {\em Math. Comp.},  {\bf 48} (1987), 521--532.
	
	\bibitem{DouDupWah74}
	J.~Douglas, Jr., T.~Dupont, and L.~Wahlbin.
	\newblock The stability in {$L\sp{q}$} of the {$L\sp{2}$}-projection into
	finite element function spaces.
	\newblock {\em Numer. Math.},  {\bf 23} (1974/75), 193--197.
	
	\bibitem{Gar97}
	D.~A. Garanin.
	\newblock {F}okker-{P}lanck and {L}andau-{L}ifshitz-{B}loch equations for
	classical ferromagnets.
	\newblock {\em Phys. Rev. B},  {\bf 55} (1997), 3050--3057.
	
	\bibitem{GolJiaLe25}
	B.~Goldys, C.~Jiao, and K.-N. Le.
	\newblock {Numerical method and error estimate for stochastic
		{L}andau–{L}ifshitz–{B}loch equation}.
	\newblock {\em IMA J. Numer. Anal.},  {\bf 45} (2025), 1821--1867.
	
	\bibitem{GuiLiWan22}
	X.~Gui, B.~Li, and J.~Wang.
	\newblock Convergence of renormalized finite element methods for heat flow of
	harmonic maps.
	\newblock {\em SIAM J. Numer. Anal.},  {\bf 60} (2022), 312--338.
	
	\bibitem{HouQia23}
	D.~Hou and Z.~Qiao.
	\newblock An implicit-explicit second-order {BDF} numerical scheme with
	variable steps for gradient flows.
	\newblock {\em J. Sci. Comput.},  {\bf 94} (2023), Paper No. 39, 22.
	
	\bibitem{Le16}
	K.~N. Le.
	\newblock Weak solutions of the {L}andau-{L}ifshitz-{B}loch equation.
	\newblock {\em J. Differential Equations},  {\bf 261} (2016), 6699--6717.
	
	\bibitem{LeSoeTra24}
	K.-N. Le, A.~L. Soenjaya, and T.~Tran.
	\newblock The {L}andau--{L}ifshitz--{B}loch equation in polytopal domains:
	Unique existence and finite element approximation.
	\newblock {\em IMA J. Numer. Anal.} (2026), \href{https://doi.org/10.1093/imanum/drag002}{https://doi.org/10.1093/imanum/drag002}.
	
	\bibitem{LeyLi21}
	D.~Leykekhman and B.~Li.
	\newblock Weak discrete maximum principle of finite element methods in convex
	polyhedra.
	\newblock {\em Math. Comp.},  {\bf 90} (2021), 1--18.
	
	\bibitem{LiShe20}
	X.~Li and J.~Shen.
	\newblock On a {SAV}-{MAC} scheme for the {C}ahn-{H}illiard-{N}avier-{S}tokes
	phase-field model and its error analysis for the corresponding
	{C}ahn-{H}illiard-{S}tokes case.
	\newblock {\em Math. Models Methods Appl. Sci.},  {\bf 30} (2020), 2263--2297.
	
	\bibitem{LiSheLiu21}
	X.~Li, J.~Shen, and Z.~Liu.
	\newblock New {SAV}-pressure correction methods for the {N}avier-{S}tokes
	equations: stability and error analysis.
	\newblock {\em Math. Comp.},  {\bf 91} (2021), 141--167.
	
	\bibitem{LiWanShe22}
	X.~Li, W.~Wang, and J.~Shen.
	\newblock Stability and error analysis of {IMEX} {SAV} schemes for the
	magneto-hydrodynamic equations.
	\newblock {\em SIAM J. Numer. Anal.},  {\bf 60} (2022), 1026--1054.
	
	\bibitem{LinThoWah91}
	Y.~P. Lin, V.~Thom\'{e}e, and L.~B. Wahlbin.
	\newblock Ritz-{V}olterra projections to finite-element spaces and applications
	to integrodifferential and related equations.
	\newblock {\em SIAM J. Numer. Anal.},  {\bf 28} (1991), 1047--1070.
	
	\bibitem{PuYan22}
	X.~Pu and L.~Yang.
	\newblock Global smooth solutions for the {L}andau-{L}ifshitz-{B}loch equation
	with helicity term.
	\newblock {\em Appl. Math. Lett.},  {\bf 133} (2022), Paper No. 108215, 7.
	
	\bibitem{RanSco82}
	R.~Rannacher and R.~Scott.
	\newblock Some optimal error estimates for piecewise linear finite element
	approximations.
	\newblock {\em Math. Comp.},  {\bf 38} (1982), 437--445.
	
	\bibitem{SheXuYan18}
	J.~Shen, J.~Xu, and J.~Yang.
	\newblock The scalar auxiliary variable ({SAV}) approach for gradient flows.
	\newblock {\em J. Comput. Phys.},  {\bf 353} (2018), 407--416.
	
	\bibitem{SheXuYan19}
	J.~Shen, J.~Xu, and J.~Yang.
	\newblock A new class of efficient and robust energy stable schemes for
	gradient flows.
	\newblock {\em SIAM Rev.},  {\bf 61} (2019), 474--506.
	
	\bibitem{Soe25c}
	A.~L. Soenjaya.
	\newblock Energy-stable finite element approximation of the
	{L}andau--{L}ifshitz--{B}loch equation below the {C}urie temperature.
	\newblock {\em J. Sci. Comput.},  {\bf 104} (2025), Paper No. 50.
	
	\bibitem{Soe24}
	A.~L. Soenjaya.
	\newblock Mixed finite element methods for the
	{L}andau--{L}ifshitz--{B}aryakhtar and the regularised
	{L}andau--{L}ifshitz--{B}loch equations in micromagnetics.
	\newblock {\em J. Sci. Comput.},  {\bf 103} (2025), Paper No. 65.
	
	\bibitem{Soe25}
	A.~L. Soenjaya.
	\newblock Numerical analysis of the {L}andau--{L}ifshitz--{B}loch equation with
	spin-torques.
	\newblock {\em ESAIM Math. Model. Numer. Anal.},  {\bf 60}, no. 3 (2026), 1451--1501.
	
	\bibitem{Soe26}
	A.~L. Soenjaya.
	\newblock Strong convergence of finite element schemes for the stochastic {L}andau--{L}ifshitz--{B}loch equation.
	\newblock {\em IMA J. Numer. Anal.} (2026), \href{https://doi.org/10.1093/imanum/drag040}{https://doi.org/10.1093/imanum/drag040}.
	
	\bibitem{Soe26s}
	A.~L. Soenjaya.
	\newblock The {L}andau-{L}ifshitz-{B}loch equation with spin diffusion:
	{G}lobal strong solution and finite element approximation.
	\newblock {\em Numer. Methods Partial Differential Equations},  {\bf 42} (2026), Paper No. e70070, 24 pp.
	
	\bibitem{SoeTra23}
	A.~L. Soenjaya and T.~Tran.
	\newblock Global solutions of the {L}andau--{L}ifshitz--{B}aryakhtar equation.
	\newblock {\em J. Differential Equations},  {\bf 371} (2023), 191--230.
	
	\bibitem{WanHuaWan21}
	M.~Wang, Q.~Huang, and C.~Wang.
	\newblock A second order accurate scalar auxiliary variable ({SAV}) numerical
	method for the square phase field crystal equation.
	\newblock {\em J. Sci. Comput.},  {\bf 88} (2021), Paper No. 33, 36.
	
	\bibitem{YanYiChe24}
	J.~Yang, N.~Yi, and Y.~Chen.
	\newblock Optimal error estimates of a {SAV}-{FEM} for the
	{C}ahn-{H}illiard-{N}avier-{S}tokes model.
	\newblock {\em J. Comput. Appl. Math.},  {\bf 438} (2024), Paper No. 115577,
	28.
	
	\bibitem{ZhaYua22}
	T.~Zhang and J.~Yuan.
	\newblock Unconditional stability and optimal error estimates of {E}uler
	implicit/explicit-{SAV} scheme for the {N}avier-{S}tokes equations.
	\newblock {\em J. Sci. Comput.},  {\bf 90} (2022), Paper No. 1, 20.
	
\end{thebibliography}

\newcommand{\noopsort}[1]{}\def\cprime{$'$}
\def\soft#1{\leavevmode\setbox0=\hbox{h}\dimen7=\ht0\advance \dimen7
	by-1ex\relax\if t#1\relax\rlap{\raise.6\dimen7
		\hbox{\kern.3ex\char'47}}#1\relax\else\if T#1\relax
	\rlap{\raise.5\dimen7\hbox{\kern1.3ex\char'47}}#1\relax \else\if
	d#1\relax\rlap{\raise.5\dimen7\hbox{\kern.9ex \char'47}}#1\relax\else\if
	D#1\relax\rlap{\raise.5\dimen7 \hbox{\kern1.4ex\char'47}}#1\relax\else\if
	l#1\relax \rlap{\raise.5\dimen7\hbox{\kern.4ex\char'47}}#1\relax \else\if
	L#1\relax\rlap{\raise.5\dimen7\hbox{\kern.7ex
			\char'47}}#1\relax\else\message{accent \string\soft \space #1 not
		defined!}#1\relax\fi\fi\fi\fi\fi\fi}

\end{document}